\documentclass[11pt]{amsart}
\usepackage[utf8]{inputenc}
\usepackage{amsmath,amssymb, mathrsfs}
\usepackage{bbm}
\usepackage{lipsum}

\usepackage{centernot}%negare simboli meglio
\usepackage{cite}
\usepackage{xcolor}

\usepackage{tikz}
\usetikzlibrary{arrows}
\usetikzlibrary{patterns}
\usetikzlibrary{shapes}
\usepackage{mathrsfs}
\usepackage{bbm}
\usepackage{enumerate}
\usepackage[shortlabels]{enumitem} %per fare riferimenti agli item di enumerate. usare sempre ref{}
\usepackage{tikz}
\usepackage{tikzcd}

\usepackage{amssymb}
\usepackage{tikz-cd}
\usetikzlibrary{cd}

\newcommand{\R}{\mathbb{R}}
\newcommand{\N}{\mathbb{N}}

\newcommand{\de}{\partial}
\renewcommand{\-}{\backslash}
\newcommand{\D}{\mathbb{D}}

\renewcommand{\a}{\alpha}

\newcommand{\f}{\varphi}
\newcommand{\e}{\varepsilon}
\newcommand{\w}{{\omega}}

\newcommand{\transv}{\mathrel{\text{\tpitchfork}}}
\makeatletter
\newcommand{\tpitchfork}{%
  \raise-0.1ex\vbox{
    \baselineskip\z@skip
    \lineskip-.52ex
    \lineskiplimit\maxdimen
    \m@th
    \ialign{##\crcr\hidewidth\smash{$-$}\hidewidth\crcr$\pitchfork$\crcr}
  }%
}
\makeatother

\newcommand{\PP}{\mathbb{P}}

\newcommand{\Prob}{(\Omega, \mathfrak{S},\P)}
\renewcommand{\P}{\mathbb{P}}
\newcommand{\E}{\mathbb{E}}
\newcommand{\nrw}{\Rightarrow}
\newcommand{\spt}{\text{supp}}
\newcommand{\g}[3]{\mathcal{G}^{#1}(#2,\R^{#3})}

\newcommand{\Cr}[3]{\mathcal{C}^{#1}(#2,\R^{#3})}
\newcommand{\nrm}[3]{\|{#1}\|_{#2,#3}}
\newcommand{\G}{\mathscr{G}}
\newcommand{\K}{\mathcal{K}}

\usepackage{hyperref}
\hypersetup{
	colorlinks=true,       % false: boxed links; true: colored links
	linkcolor=blue,         
	citecolor=blue}
\DeclareRobustCommand{\SkipTocEntry}[5]{}%Per cancellare voce di una section/subsection/chapter dalla Tablo Of Contents, far precedere da \addtocontents{toc}{\SkipTocEntry}
 \usepackage[usenames,dvipsnames]{pstricks}
 \usepackage{epsfig}
 \usepackage{pst-grad} % For gradients
 \usepackage{pst-plot} % For axes

\newtheorem{thm}{Theorem}
\newtheorem{lemma}[thm]{Lemma}
\newtheorem{cor}[thm]{Corollary}
\newtheorem{prop}[thm]{Proposition}

\newtheorem*{stella}{($*$)}

\theoremstyle{definition}

\newtheorem{defi}[thm]{Definition}
\newtheorem{remark}[thm]{Remark}

\newtheorem{example}[thm]{Example}
\newcommand{\be}{\begin{equation}}
\newcommand{\ee}{\end{equation}}

\usepackage{mathtools} %scommentare per far numerare solo le equazioni che sono chiamate nel testo
\mathtoolsset{showonlyrefs,showmanualtags} %scommentare per far numerare solo le equazioni che sono chiamate nel testo

\numberwithin{equation}{section}

\title{Differential topology of gaussian random fields}
\author{Antonio Lerario}
\thanks{\textbf{Antonio Lerario} \href{mailto:lerario@sissa.it}{lerario@sissa.it}  SISSA, Trieste, Italy. }
\author{Michele Stecconi}
\thanks{\textbf{Michele Stecconi} \href{mailto:michele.stecconi@univ-nantes.fr}{michele.stecconi@univ-nantes.fr} Laboratoire de Mathématiques Jean Leray, Nantes University (UMR 6629 du CNRS), Nantes, France.}
\begin{document}
\begin{abstract}
Motivated by numerous questions in random geometry, given a smooth manifold $M$, we approach a systematic study of the differential topology of Gaussian random fields (GRF) $X:M\to \R^k$, that we interpret as random variables with values in $\mathcal{C}^r(M, \R^k)$, inducing on it a Gaussian measure.

When the latter is given the weak Whitney topology, the convergence in law of $X$ allows to compute the limit probability of certain events in terms of the probability distribution of the limit. This is true, in particular, for the events of a geometric or topological nature, like: ``$X$ is transverse to $W$'' or ``$X^{-1}(0)$ is homeomorphic to $Z$''.

We relate the convergence in law of a sequence of GRFs with that of their covariance structures, proving that in the smooth case ($r=\infty$), the two conditions coincide, in analogy with what happens for finite dimensional Gaussian measures. We also show that this is false in the case of finite regularity ($r\in\N$), although the convergence of the covariance structures in the $\mathcal{C}^{r+2}$ sense is a sufficient condition for the convergence in law of the corresponding GRFs in the $\mathcal{C}^r$ sense.

% We endow the set of GRFs with the narrow topology and we prove results relating the convergence in the Whitney $C^\infty$ topology of the covariance structure of $X$ and the random variable $X\in C^{\infty}(M, \R^k)$ itself. When dealing with a convergent family $\{X_d\}_{d\in \mathbb{N}}$ of GRFs, these results allow to compute the limit probabilities of a family of events in terms of the probability distribution of the limit GRF. 

We complement this study by proving an important technical tools: an infinite dimensional, probabilistic version of the Thom transversality theorem, which ensures that, under some conditions on the support, the jet of a GRF is almost surely transverse to a given submanifold of the jet space.
%the second is a generalization of the Kac-Rice formula for transversal intersections, which allows to count the cardinality of the transversal preimage of a manifold under a random map. (This result is formulated in a very general framework, allowing to consider the case of the preimage under the jet map of a random field of a submanifold of the jet space.)
%
\end{abstract}
\maketitle
\addtocontents{toc}{\SkipTocEntry}
\subsection*{Keywords} Random geometry $\cdot$ Gaussian Measures $\cdot$ Smooth Random Fields $\cdot$ Limit Probabilities $\cdot$ Narrow Topology $\cdot$ Transversality.
\addtocontents{toc}{\SkipTocEntry}
\subsection*{Declaration}
This research was conducted while the second author was a PhD student at SISSA, Trieste, Italy, supported by the SISSA PhD Fellowship in Geometry and Mathematical Physics.  %At the moment of the final revision the second author was supported by the ANR postdoctoral fellowship "Symplectic, real, and tropical aspects of enumerative geometry – ENUMGEOM".
\tableofcontents
\section{Introduction}
\subsection{Overview}The subject of \emph{Gaussian random fields} is classical and largely developed (see for instance\footnote{This list is by no means complete!} \cite{AdlerTaylor, NazarovSodin2, EdelmanKostlan95, bogachev}). Motivated by problems in differential topology, in this paper we adopt a point of view which complements the classical one and we view Gaussian random fields as random variables in the space of smooth maps. Inside this space there is a rich structure coming from the geometric conditions that we can impose on the maps we are studying. There are some natural events, described by differential properties of the maps under consideration (e.g. being transverse to a given submanifold; having a certain number of critical points;  having a fixed homotopy type for the set of points satisfying some regular equation written in term of the field...), which are of specific interest to differential topology and it is desirable to have a verifiable notion of convergence of Gaussian random fields which ensures the convergence of the probability of these natural events. At the same time, once the space of functions is endowed with a probability distribution, it is natural to investigate the stability of these properties using the probabilistic language (replacing the notion of ``generic'' from differential topology with the notion of ``probability one''). 

The purpose of this paper is precisely to produce a general framework for investigating this type of questions. Specifically, Theorem \ref{thm:2} below allows to study the limit probabilities of these natural events for a family of Gaussian random fields (the needed notion of convergence is ``verifiable'' because it is written in terms of the convergence of the covariance functions of these fields). Theorem \ref{thm:1} relates this notion to the convergence of the fields in an appropriate topology: we achieve this by proving that there is a topological embedding of the set of smooth Gaussian random fields into the space of covariance functions. The switch from ``generic'' to ``probability one'' happens with Theorem \ref{transthm2}, which gives a probabilistic version of the Thom Transversality Theorem (again the needed conditions for this to hold can be checked using the covariance function of the field). This is actually a corollary of the more general Theorem \ref{thm:transthm}, which provides an infinite dimensional, probabilistic version, of the Parametric Transversality Theorem.

\subsection{Topology of random maps} Let $M$ be a smooth $m$-dimensional manifold (possibly with boundary). We denote by $E^r=\mathcal{C}^r(M, \R^k)$ the space of differentiable maps endowed with the weak Whitney topology, where $r\in \N\cup\{\infty\}$, and we call $\mathscr{P}(E^r)$ the set of probability measures on $\mathcal{C}^r(M, \R^k)$, endowed with the narrow topology (i.e. the weak* topology of $\mathcal{C}_b(E^r)^*$, see Definition \ref{defi:narrowtop}). 

In this paper we are interested in a special subset of $\mathscr{P}(E^r)$, namely the set $\G(E^r)$ of \emph{Gaussian measures}: these are probability measures with the property that for every finite set of points $p_1, \ldots, p_j\in M$ the evaluation map $\varphi:C^r(M, \R^k)\to \R^{jk}$ at these points induces (by pushforward) a Gaussian measure on $\R^{jk}.$ \footnote{In remark \ref{rem:bridge} we explain how this definition is equivalent to that of a Gaussian measure on the topological vector space $E^r$.} We denote by $\mathcal{G}^r(M, \R^k)$ the set of $\mathcal{C}^r$ \emph{Gaussian random fields} (GRF) i.e. random variables with values in $E^r$ that induce a Gaussian measure (see Definition \ref{def:GRF} below).
\begin{example}
The easiest example of GRF is that of a random function of the type $X=\xi_1 f_1 +\dots +\xi_n f_n$, where $\xi_1,\dots,\xi_n$ are independent Gaussian variables and $f_i\in \Cr rMk$. A slightly more general example is an almost surely convergent series
\be\label{eq:KarLo}
X=\sum_{n=0}^\infty \xi_n f_n.
\ee 
In fact, a standard result in the general theory of Gaussian measures (see \cite{bogachev}) is that every GRF admits such a representation, which is called the Karhunen-Loève expansion. We give a proof of this in Appendix \ref{app:rgrf} (see Theorem \ref{thm:rep}), adapted to the language of the present paper.
\end{example}
\begin{remark}One can define a Gaussian random section of a vector bundle $E\to M$ in an analogous way (the evaluation map here takes values in the finite dimensional vector space $E_{p_1}\oplus\dots\oplus E_{p_j}$). We choose to discuss only the case of  trivial vector bundles to avoid a complicated notation, besides, any vector bundle can be linearly embedded in a trivial one, so that any Gaussian random section can be viewed as a Gaussian random field. For this reason, the results we are going to present regarding GRFs are true, mutatis mutandis, for Gaussian random sections of vector bundles.
\end{remark}

We have the following sequence of continuous injections:
\be \G(E^{\infty})\subset \cdots \subset \G(E^r)\subset \cdots \subset \G(E^0)\subset \mathscr{P}(E^0),\ee
with the topologies induced by the inclusion $\G(E^r)\subset  \mathscr{P}(E^r)$ as a closed subset.

By definition, a Gaussian random field $X$ induces a Gaussian measure on $\mathcal{C}^r(M, \R^k)$, measure that we denote by $[X]$. Two fields are called \emph{equivalent} if they induce the same measures. A Gaussian measure $\mu=[X]\in \G(E^r)$ gives rise to a differentiable function $K_\mu\in \mathcal{C}^{r}(M\times M, \R^{k\times k})$ called the \emph{covariance function} and defined for $p,q\in M$ by:
\be K_\mu(p,q)=\E \left\{X(p) X(q)^T\right\}=\int_{E^r} f(p)f(q)^Td\mu(f).\ee
Equivalent fields give rise to the same covariance function, and to every covariance function there corresponds a unique (up to equivalence) Gaussian field. 

\begin{remark}In this paper we are interested in random fields up to equivalence, this is the reason why we choose to focus on the narrow topology. Indeed the notion of narrow convergence of a family $\{X_d\}_{d\in \mathbb{N}}$ of GRFs corresponds to the notion of \emph{convergence in law} of random elements in a topological space and it regards only the probability measures $[X_d]$. By the Skorohod theorem (see \cite[Theorem 6.7]{Billingsley}) this notion corresponds to almost sure convergence, up to equivalence of GRFs. In case one is interested in the almost sure convergence or in the convergence in probability of a particular sequence of GRFs one should be aware that these two notions take into account also the joint probabilities. For example, convergence in probability is equivalent to narrow convergence of the couple $(X_d,X)\nrw (X,X)$ (see Theorem \ref{lem:convprob}).
\end{remark}
Our first theorem translates convergence in $\G(E^r)$ of Gaussian measures
with respect to the narrow topology in terms of the corresponding sequence of covariance functions in the space $\mathcal{C}^{r}(M\times M, \R^{k\times k})$, endowed with the weak Whitney topology. 

\begin{thm}[Measure-Covariance]\label{thm:1}The natural map
\be\label{eq:covariance} \mathcal{K}^r:\G(E^r)\to \mathcal{C}^{r}(M\times M, \R^{k\times k}),\ee
given by $\mathcal{K}^r:\mu\mapsto K_\mu,$ is injective and continuous for all $r\in \N\cup\{ \infty\}$; when $r=\infty$ this map is also a closed topological embedding\footnote{A continuous injective map that is an homeomorphism onto its image.}.
\end{thm}
%\mik{Da valutare se mettere già $\mathcal{C}^{r,r}$. Il motivo è che leggendo l'introduzione così, può sembrare che il motivo per cui il caso $r<\infty$ non funzia è che semplicemente non stiamo prendendo il codominio giusto.}

We observe at this point that the condition $r=\infty$ in the second part of the statement of Theorem \ref{thm:1} is necessary: as Example \ref{escempio} and Theorem \ref{thm:counter} show, it is possible to build a family of $\mathcal{C}^r$ ($r\neq \infty$) GRFs with covariance structures which are $\mathcal{C}^r$ converging but such that the family of GRFs does not converge narrowly to the GRF corresponding to the limit covariance.

Theorem \ref{thm:1} is especially useful when one has to deal with a family of Gaussian fields depending on some parameters, as it allows to infer asymptotic properties of probabilities on $\mathcal{C}^\infty(M, \R^k)$ from the convergence of the covariance functions (notice that this ``implication'' goes the opposite way of the arrow in \eqref{eq:covariance}). 

\begin{thm}[Limit probabilities]\label{thm:2}Let $\{X_d\}_{d\in \mathbb{N}}\subset \mathcal{G}^{r}(M,\R^k)$ be a sequence of Gaussian random fields such that the sequence $\{K_d\}_{d\in \mathbb{N}}$ of the associated covariance functions converges to $K$ in $\mathcal{C}^{r+2,r+2}(M\times M, \R^{k\times k})$\footnote{It is the space of functions $K(x,y)$ having continuous partial derivatives of order at least $r$ in both variables $x$ and $y$, see Section \ref{sec:spasmo}.}. Then there exists $X\in \g rMk$ with $K_X=K$ such that for every Borel set $A\subset E^r$ we have
\be\label{eq:limitprob} \PP(X\in \mathrm{int}(A))\leq \liminf_{d\to \infty}\PP(X_d\in A)\leq \limsup_{d\to \infty}\PP(X_d\in A)\leq \PP(X\in \overline{A}).\ee
In particular, if $\P(X\in\de A)=0$, then the limit exists:
\be\label{eq:limitprob2}
\lim_{d\to \infty}\P(X_d\in A)=\P(X\in A).
\ee
\end{thm}

%\begin{remark}
%The notion of narrow convergence of a family $\{X_d\}_{d\in \mathbb{N}}$ of GRFs corresponds to the notion of \emph{convergence in law} of random elements in a topological space and it regards just the probability measures $[X_d]$. By Skorohod Theorem (see \cite[Theorem 6.7]{Billingsley}) this notion corresponds to almost sure convergence up to equivalence of GRFs. In case one is interested in the almost sure convergence or in the convergence in probability of a particular sequence of GRFs one should be aware that those two notions take into account also the joint probabilities, for example convergence in probability is equivalent to narrow convergence of the couple $(X_d,X)\nrw (X,X)$ (see Theorem \ref{thm:convprob}).
%\end{remark}

\subsection{The support of a Gaussian random field}
The previous Theorem \ref{thm:2} raises two natural questions: 
\begin{enumerate}[$(1)$]
\item\label{itm:Q1} when is the leftmost probability in \eqref{eq:limitprob} strictly positive?
\item\label{itm:Q2} For which sets $A\subset E^\infty$ do we have  \eqref{eq:limitprob2}?
\end{enumerate}
Answering question \ref{itm:Q1} for a given Gaussian random field $X\in \mathcal{G}^r(M, \R^k)$, amounts to determine its topological support:
\be\label{eq:defsupport} \textrm{supp}(X)=\{\textrm{$f\in E^r$ such that $\PP(X\in U)>0$ for every neighborhood $U$ of $f$}\}.\ee

We provide a description of the support of a Gaussian field $X=(X^1, \ldots, X^k)\in \mathcal{G}^{r}(M, \R^k)$ in terms of its covariance function $K_X$. 
 \begin{thm}\label{thm:dtgrf:3}
Let $X\in \g rMk$, consider all functions $h_p^j\in E^r$ of the form 
\be \label{suppbasiseq}
\begin{aligned}
\left. h^j_p(q) \right. = \begin{pmatrix}
K_X(q,p)^{1j} \\
\vdots \\
K_X(q,p)^{kj}
\end{pmatrix},\quad \text{for $p\in M$ and $j\in\{1,\dots,k\}$}. 
\end{aligned}
\ee
then
\be 
\spt(X)=\overline{\textrm{span}\{h_p^j\colon p\in M,j=1\dots k\}}^{E^r}.
\ee
\end{thm}
In particular, note that the support of a GRF is always a vector space, thus any neighborhood of $0$ has positive probability. 

Theorem \ref{thm:dtgrf:3} is just a general property of Gaussian measures, translated into the language of the present paper. In Section \ref{sec:PT3}, we prove it as a consequence of \cite[Theorem 3.6.1]{bogachev} together with a description of the Cameron-Martin space of $X$. 
In Appendix \ref{app:rgrf}, we present a direct proof of such result, adapted to our language (see Corollary \ref{cor:sputoCMapp}). We do this by generalizing the proof given in \cite[Section A.3-A.6]{NazarovSodin2} for the case in which $r=0$ and $k=1$. 

\subsection{Differential topology from the random point of view} Addressing question \ref{itm:Q2} above, let us observe that the probabilities in \eqref{eq:limitprob} are equal if and only if $\PP(X\in\partial A)=0$, and the study of this condition naturally leads us to the world of Differential Topology.

%When studying smooth maps, most relevant sets are given imposing some conditions on their jets (this is what happens, for example, when studying a given singularity class). For example, let us take for $A\subset E^\infty$ in Theorem \ref{thm:2} an open set defined by a condition on the $r$-th jet of $X$:
%\be A=\{\textrm{$f\in E^\infty$ such that $j_x^r f\in V\subseteq J^{r}(M, \R^k)$ for all $x\in M$}\}.\ee
 %A frequent condition is the case when the boundary of $A$ consists of jets which \emph{are not} transverse to a given submanifold $W\subset J^r(M, \R^k)$, and then the problem of having the equality in \eqref{eq:limitprob} reduces to $\P(j^r X\pitchfork W)=1$. 
 When studying smooth maps, most relevant sets are given imposing some conditions on their jets (this is what happens, for instance, when studying a given singularity class). For example, let us take for $A\subset E^\infty$ in Theorem \ref{thm:2} a set defined by a condition on the $r$-th jet of $X$:
\be A=\{\textrm{$f\in E^\infty$ such that $j_x^r f\in V\subseteq J^{r}(M, \R^k)$ for all $x\in M$}\}.\ee
One can show that if $V$ is an open set with smooth boundary $\de V$, then there is no map $f\in \de A$ satisfying $j^rf \pitchfork \de V$. This is a frequent situation, indeed in most cases, the boundary of $A$ consists of functions whoose jet \emph{is not} transverse to a given submanifold $W\subset J^r(M, \R^k)$, and then the problem of proving the existence of the limit \eqref{eq:limitprob2} reduces to show that $\P(j^r X\pitchfork W)=1$.  Motivated by this, we prove the following.
\begin{thm}\label{transthm2}
Let $X\in \g\infty Mk$ and denote $F=\spt(X)$. Let $r\in \N$. Assume that for every $p\in M$ we have
\be \label{eq:trans}
\mathrm{supp}(j^r_pX)=J^r_p(M,\R^k)\ee
Then for any submanifold $W\subset J^r(M,\R^k)$, we have $\P(j^rX \pitchfork W)=1$.
\end{thm}
Let us explain condition \eqref{eq:trans} better. Given $X\in \mathcal{G}^r(M, \R^k)$ and $p\in M$ one can consider the random vector $j_p^rX\in J_p^r(M, \R^k)$: this is a Gaussian variable and \eqref{eq:trans} is the condition that the support of this Gaussian variable is the whole $J_p^r(M, \R^k)$.
For example, if the support of a $\mathcal{C}^r$-Gaussian field $X$ equals the whole $E^r$, then for every $W\subset J^{r}(M, \R^k)$ we have $X\pitchfork W$ with probability one. 

We will actually prove Theorem \ref{transthm2} as a corollary of the following more general theorem, that is an infinite dimensional version of the Parametric Transversality Theorem \ref{parametric transversality}.
\begin{thm}[Probabilistic transversality]\label{thm:transthm}
Let $X\in \g r Mk$, for $r\in\N\cup \{\infty\}$, and denote $F=\spt(X)$. Let $P,N$ be smooth manifolds and $W\subset N$ a submanifold. Assume that $\Phi\colon P\times F\to N$ is a smooth map  such that $\Phi\transv W$. Then
\be 
\P\{\phi(X)\transv W\}=1,
\ee
where $\phi(f)$ is the map $p\mapsto\Phi(p,f)$.
\end{thm}
\begin{remark}
We stress the fact that the space $F$ in Theorem \ref{thm:transthm} might be infinite dimensional. This is remarkable in view of the fact that the proof of the finite dimensional analogue of Theorem \ref{thm:transthm} makes use of Sard's theorem, which is essentially a finite dimensional tool. In fact, such result is false in general for smooth maps defined on an infinite dimensional space (see \cite{counterSardKupka}). In this context, the only alternative tool is the Sard-Smale theorem (see \cite{smalesard}), which says that the set of critical values of a smooth Fredholm map between Banach spaces is meagre (it is contained in a countable union of closed sets with empty interior). However, this is not enough to say something about the evaluation of a Gaussian measure on such set, not even when the dimension is finite.

Moreover, in both the proof of the finite dimensional transversality theorem and of Sard-Smale theorem an essential instrument is the Implicit Function theorem. Although this result, in its generalized version developed by Nash and Moser, is still at our disposal in the setting of Theorem \ref{thm:transthm} (at least when $M$ is compact), it fails to hold in the context of more general spaces.

 That said, the proof of theorem \ref{thm:transthm} relies on finite dimensional arguments and on the Cameron-Martin theorem (see \cite[Theorem 2.4.5]{bogachev}), a result that is specific to Gaussian measures on locally convex spaces. In fact, the careful reader can observe that the only property of $X$ that we use is that $[X]$ is a nondegenerate Radon Gaussian measure (in the sense of \cite[Def. 3.6.2]{bogachev}) on the second-countable, locally convex vector space $F$.
\end{remark}

\section{Preliminaries}\label{sec:preliminaries}
\subsection{Space of smooth functions}\label{sec:spasmo}
Let $M$ be a smooth manifold of dimension $m$. We will always implicitely assume that $M$ is Hausdorff and second countable, possibly with boundary. Let $k\in \N$ and $r\in\N\cup\{+\infty\}$. We will refer at the set of $\mathcal{C}^r$ functions 
\be E^r=\Cr{r}{M}{k}\ee as a topological space with the weak Whitney topology as in ~\cite{Hirsch,NazarovSodin2}. Let $Q\colon D\hookrightarrow M$ be an embedding of a compact set $D\subset \R^n$, we define for any $f\in \Cr rMk$, the seminorm
\be\label{eq:seminormQR}
\|f\|_{Q,r}\doteq \sup\Big\{\big|\de_\a \big(f\circ Q\big) (x)\big| \colon \a\in \N^{m},\ |\a|\le r,\ x\in \text{int}(D)\Big\}.
\ee
Then, for $r\in \N$ finite, the weak topology on $\Cr rMk$ is defined by the family of seminorms $\{\nrm \cdot Qr\}_Q$, while the topology on $\Cr\infty Mk$ is defined by the whole family $\{\nrm \cdot Qr\}_{Q,r}$. 
%Note that is sufficient to consider any family of embeddings $Q_l\colon D\to M$ such that $\cup_{l}\text{int}(Q_l(D))=M$.
We recall that for any $r\in\N\cup\{\infty\}$, the topological space $\Cr rMk$ is a Polish space: it is separable and homeomorphic to a complete metric space (indeed it is a Fréchet space). 
We will also need to consider the space $\Cr {r,r}{M\times M}k$ consisting of those functions that have continuous partial derivatives of order at least $r$ with respect to both the product variables. The topology on this space is defined by the seminorms
\[
\|f\|_{Q,(r,r)}\doteq \sup\Big\{\big|\de_{(\a,\beta)} \big(f\circ Q\big) (x,y)\big| \colon \a,\beta \in \N^{m},\ |\a|, |\beta|\le r,\ x,y\in \text{int}(D)\Big\},
\]
where now $Q$ varies among all product embeddings: $Q(x,y)=(Q_1(x), Q_2(y))\in M\times M$ and $Q_1, Q_2$ are embeddings of two compact sets $D_1, D_2$.
\begin{lemma}\label{Lemma:14}
Let $f,f_n\in \Cr rMk$. $f_n\to f$ in $\Cr rMk$ if and only if for any convergent sequence $p_n\to p$ in $M$, 
\be
j^r_{p_n}f_n\to j^r_{p}f \quad \text{in $J^r(M,\R^k)$}.
\ee
\end{lemma}
\begin{proof}See \cite[Chapter 2, Section 4]{Hirsch}.
\end{proof}
Given an open cover $\{U_\ell\}_{\ell\in L}$ of $M$, the restriction maps define a topological embedding $
\Cr rMk \hookrightarrow \prod_{\ell\in L} \Cr r{U_\ell}k,
$
indeed any converging sequence $p_n\to p$ belongs to some $U_\ell$ eventually. In particular, suppose that $Q_\ell\colon\D^m\hookrightarrow M$ are a countable family of embeddings of the unit $m-$disk $\D^m$ such that int$(Q_\ell(\D^m))=U_\ell$ is a covering of $M$\footnote{This is always possible in a smooth manifold without boundary, by definition, and it is still true if the manifold has boundary: if $p\in \de M$, take an embedding of the unit  disk $Q\colon\D^m  \hookrightarrow M$ such that $Q(\de \D^m)$ intersects $\de M$ in an open neighborhood of $p$, then the interior of $Q(\D^m)$, viewed as a subset of $M$, contains $p$.}. Then the maps $Q_\ell^*\colon f\mapsto f\circ Q_\ell$ define a topological embedding
\be\label{prodembballeq}
\{Q_\ell^*\}_\ell \colon \Cr rMk \hookrightarrow \left(\Cr r{\D^m}k\right)^L
\ee

We refer to the book \cite{Hirsch} for the details about topologies on spaces of differentiable functions.
\subsection{Gaussian random fields}
Most of the material in this section, can be found in the book \cite{AdlerTaylor} and in the paper \cite{NazarovSodin2}; we develop the language in a slightly different way so that it suits our point of view focused on measure theory.

A real random variable $\gamma$ on a probability space $\Prob$ is said to be Gaussian if there are real numbers $\mu\in \R$ and $\sigma\geq 0$, such that $\gamma\sim N(\mu,\sigma^2)$, meaning that it induces the $N(\mu,\sigma^2)$ measure on the real numbers, which is $\delta_{\mu}$ if $\sigma=0$, and for $\sigma>0$ it has density 
\[
\rho(t)=\frac1{\sqrt{2\pi\sigma^2}}e^{-\frac{(t-\mu)^2}{2\sigma^2}}.
\]
In this paper, unless otherwise specified, all Gaussian variables and vectors are meant to be \emph{centered}, namely with $\mu=0$.

A (centered) Gaussian random vector $\xi$ in $\R^k$ is a random variable on $\R^k$ s.t. for any covector $\lambda\in(\R^k)^*$, the real random variable $\lambda \xi$ is (centered) Gaussian. In this case we write $\xi\sim N(0,K)$ where $K=\E\{\xi \xi^T\}$ is the so called \emph{covariance matrix}.
If $\xi$ is a Gaussian random vector in $\R^k$, there is a random vector $\gamma\sim N(0,\mathbbm{1}_j)$ in $\R^j$ and an injective $k\times j$ matrix $A$ s.t.
\[
\xi=A\gamma.
\]
In this case $K=AA^T$ and the support of $\xi$ is the image of $A$, which concides with the image of the matrix $K$, that is
\[
\textrm{supp}(\xi)=\{p\in\R^k \colon \P\{U_p\}>0 \text{  for all neighborhoods $U_p\ni p$ }\}=\text{Im}K\footnote{This is the finite dimensional version of Theorem \ref{thm:dtgrf:3}.},
\]
indeed $\xi\in \text{Im}K=\text{Im}A$ with $\P=1$. If $A$ is invertible, $\xi$ is said to be nondegenerate, this happens if and only if $\det K\neq 0$, if and only if $\textrm{supp}(\xi)=\R^n$, if and only if the probability induced by $\xi$ admits a density, which is given by the formula
\be\label{gaussianlaweq}
\P\{\xi\in U\}=\frac1{(2\pi)^\frac{n}{2}\det K^\frac12}\int_{U}e^{-\frac12 W^TK^{-1}W}dW^n.
\ee

%%%ciak
\begin{defi}[Gaussian random field]\label{def:GRF}
Let $M$ be a smooth manifold.
Let $\Prob$ be a probability space. An $\R^k$-valued \emph{random field (RF)} on $M$ is a measurable map
\[
X:\Omega\to {(\R^{k})}^M,
\]
with respect to the product $\sigma-$algebra on the codomain.
An $\R$-valued \emph{RF} is called a \emph{random function}.

Let $r\in\N\cup \{\infty\}$. We say that $X$ is a $\mathcal{C}^r$ random field, if $X_{\w} \in \Cr{r}{M}{k}$ for $\P$-almost every $\w\in \Omega$. 
We say that $X$ is a \emph{Gaussian random field (GRF)}, or just \emph{Gaussian field}, if for any finite collection of points $p_1,\dots, p_j\in M$, the random vector in $\R^{jk}$ defined by $(X(p_1),\dots,X(p_j))$ is Gaussian.
We denote by $\mathcal{G}^r(M, \R^k)$ the set of $\mathcal{C}^r$ Gaussian fields.
\end{defi}
When dealing with random fields $X:\Omega\to (\R^k)^M$, we will most often use the shortened notation of omitting the dependence from the variable $\w$. In this way $X:M\to\R^k$ is a  \emph{random} map, i.e. a random element\footnote{We recall that, given a measurable space $(S,\mathfrak{A})$, a measurable map from a probability space $\Prob$ to $S$ is also called a \emph{Random Element} of $S$ (see \cite{Billingsley}). Random variables and random vectors are random elements of $\R$ and $\R^k$, respectively.} of $(\R^k)^M$.

\begin{remark}
In the above definition, the sentence:
\be \textrm{``$X_{\w} \in \Cr{r}{M}{k}$ for $\P$-almost every $\w\in \Omega$''}\ee 
means that the set $\{\w\in\Omega\colon X_\w\in \Cr{r}{M}{k}\}$ contains a \emph{measurable} set $\Omega_0$ which has probability one. We make this remark because the subset $\Cr rMk$ doesn't belong to the product $\sigma-$algebra of $(\R^{k})^M$. 
\end{remark}

\begin{lemma}\label{borelemm1}
For all $r \in \N\cup\{+\infty\}$ the Borel $\sigma$-algebra $\mathcal{B}\Big(\Cr rMk\Big)$ is generated by the sets 
\[
\{f\in\Cr rMk \ :\ f(p)\in A\}
\]
with $p\in M$ and $A\subset \R^k$ open. 
Moreover $\Cr rMk$ is a Borel subset of $\Cr 0Mk$, for all $r \in \N\cup\{+\infty\}$.\end{lemma}
\begin{proof}See \cite[p. 43,44]{NazarovSodin2} or \cite[p. 374]{bogachev}.
\end{proof}

As a consequence we have that the Borel $\sigma$-algebra $\mathcal{B}(\Cr rMk)$ is  the restriction to $\Cr rMk$ of the product $\sigma$-algebra of $(\R^k)^M$. It follows that $X$ is a $\mathcal{C}^r$ RF on $M$ if and only if it is $\P-$almost surely equal to a random element of $\Cr rMk$.

A second consequence is that if $X$ is a $\mathcal{C}^r$ RF, then the associated map $\tilde{X}\colon \Omega\times M\to \R^k$ is measurable, being the composition $e\circ (X\times \text{id})$, where $e\colon \Cr rMk \times M \to \R^k$ is the continous map defined by $e(f,p)=f(p)$.

%%%ciak 1

%\begin{defi}[Gaussian Random Field]\label{def:GRF}
%Let $M$ be a smooth manifold.
%Let $\Prob$ be a probability space. An $\R^k$-valued \emph{Random Field (RF)} on $M$ is a measurable map
%\[
%X:\Omega\times M\to \R^{k}.
%\]
%An $\R$-valued \emph{RF} is called a \emph{Random Function}.
%\comm{At first it seemed that the definition given in this way was simpler, but now I might want to go back to the old definition}
%Let $r\in\N\cup{\{\infty\}}$. We say that $X$ is a $\mathcal{C}^r$ field, if $X(\w, \cdot) \in \Cr rMk$ for $\P$-almost every $\w\in \Omega$. 
%We say that $X$ is \emph{Gaussian (GRF)} if for any finite collection of points $p_1,\dots, p_j\in M$, the random vector in $\R^{jk}$ defined by $(X(\cdot , p_1),\dots,X(\cdot,p_j))$ is gaussian.
%We denote by $\mathcal{G}^r(M, \R^k)$ the set of $\mathcal{C}^r$ gaussian fields.
%\end{defi}

%%%ciak 2

If $X$ is a $\mathcal{C}^r$ RF, then it induces a probability measure $X_*\P$ on $\Cr rMk$, or equivalently (because of Lemma \ref{borelemm1}) a probability measure on $\Cr 0Mk$ that is supported on $\Cr rMk$. We say that two RFs are \emph{equivalent} if they induce the same measure; note that this can happen even if they are defined on different probability spaces.

It is easy to see that every probability measure $\mu$ on $\Cr rMk$ is induced by some RF (just take $\Omega=\Cr rMk$, $\mu=\P$ and define $X$ to be the identity, then clearly $\mu=X_*\P$). This means that the study of $\mathcal{C}^r$ random fields up to equivalence corresponds to the study of Borel probability measures on $\Cr rMk$. 

Note that, as a consequence of Lemma \ref{borelemm1}, a Borel measure $\mu$ on $\Cr rMk$ is uniquely determined by its \emph{finite dimensional distributions}, which are the measures induced on $\R^{kj}$ by evaluation on $j$ points.

We will write $\mu=[X]$ to say that the probability measure $\mu$ is induced by a random field $X$.
In particular we define the \emph{Gaussian measures} on $\Cr rMk$ to be those measures that are induced by a $\mathcal{C}^r$ GRF, equivalently we give the following measure-theoretic definition.
\begin{defi}[Gaussian measure]\label{def:gauss}
Let $M$ be a smooth manifold and let $r\in \N\cup \{\infty\}$, $k\in \N$.
A \emph{Gaussian measure} on $\Cr rMk$ is a probability measure on the topological space $\Cr rMk$, with the property that for any finite set of points $p_1,\dots p_j\in M$, the measure induced on $\R^{jk}$ by the map $f\mapsto (f(p_1),\dots,f(p_j))$ is Gaussian (centered and possibly degenerate).
We denote by $\G(E^r)$ the set of Gaussian probability measures on $E^r=\mathcal{C}^r(M, \R^k).$
\end{defi}
\begin{remark}\label{rem:bridge}
In general a Gaussian measure on a topological vector space $W$ is defined as a Borel measure on $W$ such that all the elements in $W^*$ are Gaussian random variables (see \cite{bogachev}). In the case $W=\Cr rMk$, this is equivalent to Definition \ref{def:gauss}, because the set of functionals $f\mapsto a_1f(p_1)+\dots+a_jf(p_j))$ is dense in the topological dual $W^*$ (Theorem \ref{thm:deltadense} of Appendix \ref{app:dual}), therefore every continuous linear functional $\lambda\in W^*$ can be obtained as the almost sure limit of a sequence of Gaussian variables and thus it is Gaussian itself.
\end{remark}

We prove now a simple Lemma that will be needed in the following. Given a differentiable map $f\in \mathcal{C}^r(M, \R^k)$ with $r\geq 1$, and a smooth vector field $v$ on $M$, we denote by $vf$ the derivative of $f$ in the direction of $v$.

\begin{lemma}\label{vXisGRF}
Let $X\in\g rMk$ and let $v$ be a smooth vector field on $M$. Then $vX\in\g {r-1}Mk$.
(Notice that, as a consequence, the $r$-jet of a $\mathcal{C}^r$ GRF is a $\mathcal{C}^0$ GRF.)
\end{lemma}
\begin{proof}
Since $ X\in \Cr rMk$ almost surely, then $vX\in \Cr {r-1}Mk$ almost surely, thus $vX$ defines a probability measure supported on $\Cr {r-1}Mk$. To prove that it is a Gaussian measure, note that $vX(p)$ is either a $N(0,0)$ Gaussian, if $v_p=0$, or an almost sure limit of Gaussian vectors, indeed passing to a coordinate chart $x^1,\dots,x^m$ centered at $p$ s.t. $v_p=\frac{\de}{\de x^1}$, we have 
\[
vX(p)=\lim_{t\to 0}\frac{X(t,0,\dots, 0)-X(0,0,\dots, 0)}{t}\quad\text{a.s.}
\]
therefore it is Gaussian. The analogous argument can be applied when we consider a finite number of points in $M$.
\end{proof}

\subsection{The topology of random fields.} We denote by $\mathscr{P}(E^r)$, the set of all  Borel probability measures on $E^r$. We shall endow the space $\mathscr{P}(E^r)$ with the \emph{narrow topology}, defined as follows. Let $\mathcal{C}_b(E^r)$ be the Banach space of all bounded continuous functions from $E^r$ to $\R$.
\begin{defi}[Narrow topology]\label{defi:narrowtop}
The narrow topology on $\mathscr{P}(E^r)$ is defined as the coarsest topology such that for every $\varphi\in \mathcal{C}_b(E^r)$ the map $\textrm{ev}_{\varphi}:\mathcal{P}(E^r)\to \R$ given by:
\be\textrm{ev}_{\varphi}:\P \mapsto \int_{E^r} \varphi \,d\P\ee
is continuous.
\end{defi}
In other words, the narrow topology is the topology induced by the weak-$*$ topology of $\mathcal{C}_b(E^r)^*$, via the inclusion
\[
\mathscr{P}(E^r)\hookrightarrow \mathcal{C}_b(E^r)^*
\]
\[
\P\mapsto \E\{\cdot\}
\]
\begin{remark} The narrow topology is also classically refered to as the \emph{weak topology} (see \cite{Parth}, \cite{bogachev} or \cite{Billingsley}). We avoid the latter terminology to prevent confusion with the topology induced by the weak topology of $\mathcal{C}_b(E^r)^*$, which is strictly finer. Indeed if a sequence of probability measures $\mu_n$ converges to a probability measure $\mu$ in the weak topology of $\mathcal{C}_b(E^r)^*$, then for any measurable set $A\in E^r$, it holds $\lim_{n\to\infty}\mu_n(A)=\mu(A)$. This is a strictly stronger condition than narrow convergence, see Portmanteau's theorem \cite{Billingsley}.\end{remark}

\begin{center}
  $\ast$~$\ast$~$\ast$
\end{center}

Convergence of a sequence of probability measures $\mu_n\in \mathscr{P}(E^r)$ in the narrow topology is denoted as $\mu_n\nrw \mu$.
From the point of view of random fields, $[X_n]\nrw[X]$ in $\mathscr{P}(E^r)$, if and only if 
\[
\lim_{n\to \infty}\E\{\varphi(X_n)\}=\E\{\varphi(X)\} \qquad \forall \varphi\in \mathcal{C}_b(E^r)
\]
and in this case we will simply write $X_n\nrw X$.
This notion of convergence of random variables is also called convergence \emph{in law} or \emph{in distribution}. 

To understand the notion of narrow convergence it is important to recall Skorohod's theorem (see \cite[Theorem 6.7]{Billingsley}), which states that $\mu_n\nrw \mu_0$ in $\mathscr{P}(E^r)$ if and only if there is a sequence $X_n$ of random elements of $E^r$, such that $\mu_n=[X_n]$ and $X_n\to X_0$ almost surely. In other words, narrow convergence is equivalent to almost sure convergence from the point of view of the measures $\mu_n$.

However, for a given sequence of random fields $X_n$, the notion of narrow convergence is even weaker than that of convergence in probability. The subtle difference, as showed in Lemma \ref{lem:convprob} below, is that the latter takes into account the joint distributions. 
\begin{lemma}\label{lem:convprob} %needs Skorohod
Let $X_d, X \in \g r Mk$. The sequence $X_d$ convergese to $X$ in probability if and only if $(X_d,X)\nrw (X,X)$.
\end{lemma}
\begin{proof}
First, note that if $X_d\to X$ in probability, then $(X_d,X)\to (X,X)$ in probability and therefore $(X_d,X)\nrw (X,X)$. For the converse,
let $d$ be any metric on $\Cr rMk$. Since $d$ is a continuous function, if $(X_d,X)\nrw (X,X)$ then $d(X_d,X)\nrw 0$, which is equivalent to convergence in probability, by definition.
\end{proof}

We recall the following useful fact relating properties of the topology of $E$ to properties of the narrow topology on $\mathscr{P}(E);$ for the proof the reader is referred to \cite[p. 42-46]{Parth}.
\begin{prop}\label{prop:equiva}The following properties are true:
\begin{enumerate}
\item $E$ is separable and metrizable if and only if $\mathscr{P}(E)$ is separable and metrizable. In this case, the map $E\hookrightarrow \mathscr{P}(E)$, defined by $f\mapsto \delta_f$, is a closed topological embedding and the convex hull of its image is dense in $\mathscr{P}(E)$.

\item $E$ is compact and metrizable if and only if $\mathscr{P}(E)$ is compact and metrizable.

\item $E$ is Polish if and only if $\mathscr{P}(E)$ is Polish.

\end{enumerate}
%
%Let $\P\in \mathscr{P}(E)$, then $\P$ is supported on a countable union of compact subsets of $E$; for any borel set $B\in E$, it holds
%\[
%\begin{aligned}
%\P\{B\}&=\sup\{\P\{K\}\ :\ K \subset B \text{ compact}\}\\
%&=\inf\{\P\{O\}\ :\ O \supset B \text{ open}\};
%\end{aligned}
%\]
%$\mathcal{C}_b(E)$ is dense in $L^1(E,\P)$.
\end{prop}
The following corollary will be useful for us.
\begin{cor}\label{inducedmap}
Let $E_1$ and $E_2$ be two separable metric spaces. Let $\pi\colon E_1\to E_2$ be continuous. Then the induced map $\pi_*\colon \mathscr{P}(E_1)\to \mathscr{P}(E_2)$ is continuous. If moreover $\pi$ is a topological embedding, then $\pi_*$ is a topological embedding as well.
\end{cor}
\begin{proof}
If $\pi$ is continuous, then for any bounded and continuous real function $\varphi\in\mathcal{C}_b(E_2)$, the composition $\varphi\circ \pi$ is in $\mathcal{C}_b(E_1)$. Hence, the function $\int_{E_1}(\varphi\circ \pi) \colon \mathscr{P}(E_1) \to \R$ defined as $\P\mapsto \int_{E_1}(\varphi\circ \pi)d\P$ is continuous. Observe that for any $\P\in \mathscr{P}(E_1)$
\be 
\int_{E_1}(\varphi\circ \pi)d\P =\int_{E_2}\varphi \,d(\pi_*\P)=\left(\int_{E_2}\varphi\right)\circ \pi_*(\P),
\ee
thus the composition $(\int_{E_2}\varphi)\circ \pi_*\colon \mathscr{P}(E_1)\to \R$ is continuous for any $\varphi\in\mathcal{C}_b(E_2)$. From the definition of the topology on $\mathscr{P}(E_2)$, it follows that $\pi_*$ is continuous.

%Let $\P_n\nrw \P_0\in \mathcal{P}(E_1)$. Then for any open set $V\subset E_2$,
%\be 
%\liminf_n \pi_*\P_n\{V\}=\liminf_n \P_n\{\pi^{-1}(V)\}\ge \P_0\{\pi^{-1}(V)\}=\pi_*\P_0\{V\}.
%\ee
%Thus, by Portmanteu's theorem \ref{Portmanteau}, $\pi_*\P_n\nrw \pi_*\P_0$. Therefore $\pi_*$ is sequentially continuous, hence continuous because its domain and codomain are metric spaces.

Assume now that $\pi$ is a topological embedding. This is equivalent to say that $\pi$ is injective and any open set $U\subset E_1$ is of the form $\pi^{-1}(V)$ for some open subset $V\subset E_2$, and the same for Borel sets. It follows that $\pi_*$ is injective, indeed if two probability measures $\P_1, \P_2\in \mathscr{P}(E_1)$, have equal induced measures $\pi_*\P_1=\pi_*\P_2$, then 
\be 
\P_1\{\pi^{-1}(V)\}=\P_2\{\pi^{-1}(V)\}
\ee
for any Borel subset $V\subset E_2$, thus $\P_1\{U\}=\P_2\{U\}$ for any Borel subset $U\subset E_1$ and $\P_1=\P_2$.

It remains to prove that $\pi_*^{-1}$ is continuous on the image of $\pi_*$. Let $\P_n\in \mathscr{P}(E_1)$ be such that $\pi_*\P_n\nrw \pi_*\P_0$. Let $U\subset E_1$ be open, then there is some open subset $V\subset E_2$ such that $\pi^{-1}(V)=U$ and, by Portmanteau's theorem (see \cite[p. 40]{Parth}), we get
\be 
\liminf_n \P_n\{U\}=\liminf_n \pi_*\P_n\{V\}\ge \pi_*\P_0\{V\}=\P_0\{U\}.
\ee
This implies that $\P_n\nrw \P_0$. We conclude using point (1) of Proposition \ref{prop:equiva},and the fact that on metric spaces, sequential continuity is equivalent to continuity.
\end{proof}
\begin{example}
Let $\phi \colon M \to N$ be a $\mathcal{C}^r$ maps between smooth manifolds, then the map $\phi^*\colon \mathcal{C}^r(N,W)\to \mathcal{C}^r(M,W)$ defined as $\phi^*(f)=f\circ \phi$ is continuous, therefore the induced map between the spaces of probabilities, which we still denote as $\phi^*$, is continuous. The same holds for the map $\phi_*\colon \mathcal{C}^r(W,M)\to \mathcal{C}^r(W,N)$, such that $\phi_*(f)=\phi\circ f$. 
\end{example}

Note that $\mathcal{C}^r$ narrow convergence implies $\mathcal{C}^s$ narrow convergence, for every $s\le r$, but not vice versa. Indeed there are continuous injections
\be \G(E^{\infty})\subset \cdots \subset \G(E^r)\subset \cdots \subset \G(E^0)\subset \mathscr{P}(E^0).\ee

\begin{prop}
$\G(E^r)$ is closed in $\mathscr{P}(E^r)$.
\end{prop}
\begin{proof}
Let $X_n\in \g rMk$ s.t. $X_n\nrw X\in \mathscr{P}(E^r)$. Then for any $p_1\dots, p_j\in M$ we have
\[
\left(X_n(p_1),\dots ,X_n(p_j)\right)\nrw \left(X(p_1),\dots ,X(p_j)\right) \]
in $\mathscr{P}(\R^{jk})$. Therefore the latter is a Gaussian random vector and thus $[X]\in \G(E^r).$
\end{proof}
\subsection{The covariance function.}
Given a Gaussian random vector $\xi$, it is clear by equation \eqref{gaussianlaweq} that the corresponding measure $\mu=[\xi]$ on $\R^m$ is determined by the covariance matrix $K=\E\{\xi\xi^T\}$. Similarly, if $X\in \g rMk$, then $[X]$ is a measure on $\Cr rMk$ and it is uniquely determined by its finite dimensional distributions, which are the Gaussian measures induced on $\R^{kj}$ by evaluation on $j$ points. It follows that $[X]$ is uniquely determined by the collection of all the covariances of the evaluations at couples of points in $M$, which we call \emph{covariance function}.
\begin{defi}[covariance function]
Given $X\in \g rMk$, we define its \emph{covariance function} as:
\be 
K_X\colon M\times M \to \R^{k\times k} 
\ee
\be 
K_X(p,q)=\E\{X(p)X(q)^T\}.
\ee
\end{defi}
%The matrix $K_X(p,p)$ is the covariance matrix of the random vector $X(p)$, similarly for $p_1,\dots p_J$, the covariance matrix of the random vector $\xi=(X(p_1),\dots , X(p_J))$ is 
%\be 
%\begin{pmatrix}
%K_X(p_1,p_1) &&
%\end{pmatrix}
%\ee
The function $K_X$ is symmetric: $K_X(p,q)^T=K_X(q,p)$ and non-negative definite, which means that for any $p_1,\dots, p_j\in M$ and $\lambda_1,\dots, \lambda_j \in \R^k$, $
\sum_{i=1}^j \lambda_j^TK_X(p_i,p_i)\lambda_j\ge 0.$

The covariance function of a $\mathcal{C}^r$ random field is of class $\mathcal{C}^{r,r}$, see Section \ref{sec:spasmo}. This is better understood by introducing the following object.
Suppose that $X$ is a Gaussian random field on $M$, defined on a probability space $\Prob$, then it defines a map
\be\label{eq:gammaX}
 \gamma_X :M\to L^2\Prob^k\ee
such that $\gamma_X(p)=X(p)$.

To say that $X$ is Gaussian is equivalent to say that span$\{\gamma_X(M)\}$ is a Gaussian subspace of $L^2\Prob^k$, namely a vector subspace whose elements are Gaussian random vectors. 
Next proposition from \cite{NazarovSodin2} will be instrumental for us.

\begin{prop}[Lemma A.3 from \cite{NazarovSodin2}]\label{Ltwocont}
Let $X\in \g rMk$, then the map $\gamma_X\colon M\to L^2\Prob^k$ is $\mathcal{C}^r$. Moreover if $x,y$ are any two coordinate charts on $M$, then
\be 
\E\left\{\de_\a X(x)\left(\de_\beta X(y)\right)^T\right\}=
\langle \de_\a \gamma_X (x), \de_\beta \gamma_X(y)^T\rangle_{L^2\Prob}=
\de_{(\a,\beta)} K_X(x,y).
\ee
for any multi-indices $|\a|, |\beta|\le r$.
\end{prop}
\begin{cor}\label{cor:covcrr}
Let $X\in \g rMk$, then $K_X\in \Cr {r,r}{M\times M}{k\times k}$.
\end{cor}

\subsection{A Gaussian inequality}
The scope of this section is to prove Theorem \ref{Gaussinequality}, which contains a key technical inequality. Although such inequality can be seen as a consequence of Kolmogorov's theorem for $\mathcal{C}^{k,k}$ kernels, as discussed in \cite[Sec. A.9]{NazarovSodin2}, we report here a simpler proof. 
In fact, the result follows from a general inequality valid for all GRFs, not necessarily continuous. 

Given a GRF $X\colon M\to \R$, we define for all $\e>0$ the quantity $N(\e)$, to be the minimum number of $L^2$-balls of radius $\e$ needed to cover $\gamma_X(M)$. This number is always finite if $\gamma_X(M)$ is relatively compact in $L^2$. We will need the following Theorem from \cite{AdlerTaylor}.
\begin{thm}[Theorem 1.3.3 from \cite{AdlerTaylor}]\label{Adlerineq}
Let $\gamma_X(M)$ be compact in $L^2\Prob$. Let $\Delta_X=\text{diam}(\gamma_X(M))$. There exists a universal constant $C>0$ such that
\be 
\E\left\{\sup_{x\in M}X(t)\right\}\le C\int_0^{\Delta_X}\sqrt{\ln N(\e)}d\e.
\ee

\end{thm}
As a corollary, in our setting we can derive the following.
\begin{lemma}\label{zeroineqlemma}
Let $X\in \g {1}M{}$ and consider an embedding $Q: D \hookrightarrow M$ of a compact disk $D\subset \R^m$. There is a constant $C_Q>0$ such that
\be 
\E\{\|X\|_{Q,0}\}\le C_Q \sqrt{\|K_X\|_{Q\times Q,1 }}.
\ee
\end{lemma}
\begin{proof}
It is not restrictive to assume that $M=D$ and $Q=\textrm{id}$.
Notice that since the map $\gamma_X$ is continuous, by Proposition \ref{Ltwocont}, it follows that $\gamma_X(D)$ is compact in $L^2\Prob$, so that we can apply Theorem \ref{Adlerineq} to get that
\be
\E\{\|X\|_{D, 0}\}\le 2C\int_0^{\Delta_X}\sqrt{\ln N(\e)}d\e.
\ee
Moreover, for any $q,p\in D$, we have that 
\be 
\begin{aligned}
\|X(p)-X(q)\|_{L^2}^2 &=K(p,p)+K(q,q)-2K(p,q) \\
&\le \left|K(p,p)-K(q,p)\right|+\left|K(q,q)-K(p,q)\right|\\
&\le
2 \sup_{x,y\in D} \left|\frac{\de K}{\de x}(x,y)\right||p-q|,
\end{aligned}
\ee
where $K=K_X$. Thus, denoting $\Lambda^2=2\|K\|_{Q\times Q, 1}$, we obtain that 
\be\label{ineq}
\|X(p)-X(q)\|_{L^2}\le \Lambda |q-p|^{\frac12}.
\ee
Let now $\tilde{N}(\rho)$ be the minimum number of standard balls in $\R^m$ with radius $\rho$, required to cover $D$. A consequence of \eqref{ineq} is that every ball of radius $\rho$ in $D$ is contained in the preimage via $\gamma_X$ of a ball of radius $\Lambda\rho^{\frac12}$ in $L^2$, therefore $N(\e)\le \tilde{N}(\frac{\e^2}{\Lambda^2})$. Besides, $\Delta_X\le \Lambda\sqrt{R}$, where $R$ is the diameter of $D$, so that
\be 
\E\{\|X\|_{D,0}\} \le 2C\int_0^{\Lambda \sqrt{R}}\sqrt{\ln \tilde{N}\left(\frac{\e^2}{\Lambda^2}\right)}d\e 
= 
2C\Lambda\int_0^{\sqrt{R}}\sqrt{\ln \tilde{N}\left(s^2\right)}ds.
\ee
Now, since $D\subset \R^m$, there is a constant $c_m$ such that $\tilde{N}(\rho)\le c_m\left(\frac{R}{\rho}\right)^m$, therefore
\[
I(R)=\int_0^{\sqrt{R}}\sqrt{\ln \tilde{N}\left(s^2\right)}ds\le \int_0^{\sqrt{R}}\sqrt{\ln c_m\left(\frac{R^2}{s^2}\right)^{m} }ds<\infty.
\]
We conclude that $\E\{\|X\|_{D,0}\}\le 2\sqrt{2}\cdot C\cdot I(R)\sqrt{\|K_X\|_{Q\times Q,1 }}$.
\end{proof}
We are now able to prove the required Gaussian inequality.
\begin{thm}\label{Gaussinequality}
Let $X\in \g {r}Mk$ and consider an embedding $Q: D \hookrightarrow M$ of a compact disk $D\subset \R^m$. Then
\[
\E\{\|X\|_{Q, {r-1}}\}\le C \sqrt{\|K_X\|_{Q\times  Q, (r,r)}},
\]
Where $C$ is a constant depending only on $Q$, $r$ and $k$.
\end{thm}
\begin{proof}
A repeated application of Lemma \ref{vXisGRF} proves that $\de_\a X^i$ is Gaussian, so that we can use Lemma \ref{zeroineqlemma} as follows.
\be 
\begin{aligned}
\E\{\|X\|_{Q,r-1}\} &\le \sum_{|\a|< r, i\le k}\E\{\|\de_\a X^i\|_{Q,0}\} \\
&\le \sum_{|\a|< r, i\le k}C_Q \sqrt{\|K_{\de_\a X^i}\|_{Q\times Q,1}} \\
& =\sum_{|\a|< r, i\le k}C_Q \sqrt{\|\de_{(\a,\a)} K^{i,i}_{X}\|_{Q\times Q,1}} \\
&\le C(Q,r,k)\sqrt{\|K_{X}\|_{Q\times Q,(r,r)}}\ .
\end{aligned}
\ee
\end{proof}

%\input{proofs}
%
%
%
%
%
%
%
%
%
%
%
%%%%%%%%%%%%%%%%%%%%%%%%%%%%%%%%%%%%%%%%%%%%%%%%%%%%%%%%%%%%%%%%%%%
%%%%%%%%%%%%%%%%%%%%%%%%%%%%%%%%%%%%%%%%%%%%%%%%%%%%%%%%%%%%%%%%%%%
%% PROOF OF THEOREM 1 AND THEOREM 2 %%
%%%%%%%%%%%%%%%%%%%%%%%%%%%%%%%%%%%%%%%%%%%%%%%%%%%%%%%%%%%%%%%%%%%
%%%%%%%%%%%%%%%%%%%%%%%%%%%%%%%%%%%%%%%%%%%%%%%%%%%%%%%%%%%%%%%%%%%
%
%
%
%
\section{Proof of Theorem \ref{thm:1} and Theorem \ref{thm:2}}\label{sec:thm12}

\subsection{Proof that $\K^r$ is injective and continuous}We already noted that $K_X$ determines $[X]$, and this is equivalent to say that $\K^0$ is injective. It follows that $\K^r$ is injective for every $r$, since $\K^r$ is just the restriction of $\K^0$ to $\G(E^r)$. 

Let us prove continuity. Since both the domain and the codomain are metrizable topological spaces, it will be sufficient to prove sequential continuity. 
Let $\mu_n\nrw\mu\in \G(E^r)$. Let $X\in \mathcal{G}^r(M, \R^k)$ be a GRF such that $\mu=[X]$ and for every $n\in \N$ let $X_n\in \mathcal{G}^r(M, \R^k)$ be such that $\mu_n=[X_n].$  By Skorohod's representation theorem (see \cite[Theorem 6.7]{Billingsley}) we can assume that the $X_n$ are GRFs defined on a common probability space $\Prob$ and that $X_n\to X$ almost surely in the topological space $\Cr rMk$.

To prove $\mathcal{C}^{r,r}$ convergence of $K_n=K_{X_n}$ to $K=K_X$, it is sufficient (and necessary) to show that given coordinate charts $(x,y)$ on $M\times M$, a sequence $(x_n,y_n)\to (x_0, y_0)$, a couple of indices $|\a|,|\beta|\le r$ and two indices $i,j \in\{1,\dots,k\}$, then
\be\label{Kconvergeseq}
\de_{(\a,\beta)} K_n^{i,j}(x_n,y_n)\to \de_{(\a,\beta)} K^{i,j}(x_0,y_0).
\ee
Let $\gamma_n=\de_\a X_n^i(x_n)$ and $\xi_n=\de_\beta X_n^j(y_n)$. By Lemma \ref{vXisGRF}, these two random vectors are Gaussian; moreover $\gamma_n\to \gamma$ and $\xi_n\to \xi$ almost surely. It follows that the convergence holds also in $L^2\Prob$, so that
\be 
\E\{\gamma_n\xi_n\}\to \E\{\gamma\xi\},
\ee
which is exactly \eqref{Kconvergeseq}.

\subsection{Relative compactness}
As we will see with Theorem \ref{thm:counter}, the map $\K^r$ is not proper when $r$ is finite. However, we have the following partial result.
%\begin{thm}\label{rproper}
%Let $r\in \N$ and consider $[X_n]\in \G(E^{r+2})$ be such that for every $Q:D\hookrightarrow M$ embedding of a compact set $D\subset \R^m$, 
%\be 
%\sup_{n}\|K_{X_n}\|_{Q\times Q, (r+2,r+2)} <\infty .
%\ee
%Then the sequence $\{[X_n]\}_{n\in \N}$ is relatively compact in $\G(E^r)$.
%\end{thm}
\begin{thm}\label{rproper}
Let $r\in \N$ and consider $[X_n]\in \G(E^{r+2})$ and let $\{Q_\ell\}_{\ell\in \N}$ be a countable family of embeddings $Q_\ell:\D^m\hookrightarrow M$, such that the family of open sets $\textrm{int}(Q_\ell(\D))$ is a covering of $M$ (so that condition \eqref{prodembballeq} holds). Then the following conditions are related by the implications: \ref{itm:Ir} $\implies$ \ref{itm:IIr} $\implies$ \ref{itm:IIIr}.
\begin{enumerate}[(I)]
\item\label{itm:Ir} $ \sup_{n}\|K_{X_n}\|_{Q_\ell\times Q_\ell, (r+2,r+2)} <\infty$, for very $\ell\in \N$.
\item \label{itm:IIr} 
$\sup_{n}\E\left\{\|X_n\|_{Q, r+1}\right\} <\infty$, for any embedding $Q\colon D\hookrightarrow M$ of a compact $D\subset \R^m$.
\item \label{itm:IIIr}
The sequence $\{[X_n]\}_{n\in \N}$ is relatively compact in $\G(E^r)$.
\end{enumerate}
\end{thm}
Since there is a continuous inclusion $\G (E^\infty)\subset \G (E^r)$ and the function $\K^r$ is continuous, for any $r\in\N$, we immediately obtain the following corollary for the case $r=\infty$.
\begin{thm}\label{infproper}
Let $[X_n]\in \G(E^{\infty})$ and $\{Q_\ell\}_{\ell\in \N}$ be a countable family of embeddings $Q_\ell:\D^m\hookrightarrow M$, such that the family of open sets $\textrm{int}(Q_\ell(\D))$ is a covering of $M$ (so that condition \eqref{prodembballeq} holds). Then the following conditions are equivalent.
\begin{enumerate}[(I)]
\item\label{itm:Iinf} $ \sup_{n}\|K_{X_n}\|_{Q_\ell\times Q_\ell, (r,r)} <\infty$, for very $r,\ell\in \N$.
\item \label{itm:IIinf} 
$\sup_{n}\E\left\{\|X_n\|_{Q, r}\right\} <\infty$, for any embedding $Q\colon D\hookrightarrow M$ of a compact $D\subset \R^m$ and every $r\in\N$.
\item \label{itm:IIIinf}
The sequence $\{[X_n]\}_{n\in \N}$ is relatively compact in $\G(E^\infty)$.
\end{enumerate}
\end{thm}

%\begin{thm}\label{infproper}
%Let $[X_n]\in \G(E^{\infty})$ be such that for every $Q:D\hookrightarrow M$ embedding of a compact set $D\subset \R^m$ and every $r\in \N$: 
%\be 
%\sup_{n}\|K_{X_n}\|_{Q\times Q,r} <\infty .
%\ee
%Then the sequence $\{ [X_n]\}_{n\in \N}$ is relatively compact in $\G(E^\infty)$.
%\end{thm}
%Before giving the proof, we shall note that $\Cr rMk$ has the product topology with respect to the countable family of maps $\{Q_i^*\}_{i}$, defined as $Q_{i}^*(f)=f\circ Q_i\in \Cr rDk$, where $\{Q_i\}$ is a countable family $Q_i: D_i \hookrightarrow M$ of embeddings such that $\cup_n int(Q_i(D_i))=M$. Indeed a function $F\colon T\to \Cr rMk$ is continuous if and only if $Q_i^*F\colon T \to \Cr r{D_i}k$ is continuous for every $i$. 
Before proving this Theorem, recall that $\Cr rMk$ has the product topology with respect to the countable family of maps $\{Q_\ell^*\}_{\ell\in \N}$, defined as in \eqref{prodembballeq}.  It follows that a subset $\mathscr{A}\subset \Cr rMk$ is relatively compact if and only if $Q_{\ell}^*\mathscr{A}\subset \Cr r{\D^m}k$ is relatively compact for all $\ell$. In particular, if $r<\infty$, given constants $A_\ell >0$, the set
\be 
\mathscr{A}^r=\left\{f\in \Cr rMk \colon \|f\|_{Q_\ell,r+1}\le A_\ell \ \forall \ell\right\}
\ee
is compact in $\Cr rMk$. Similarly, given $A_\ell^r>0$ for all $r, \ell\in \N$, the set 
\be 
\mathscr{A^\infty}=\left\{f\in \Cr \infty Mk \colon \|f\|_{Q_\ell,r}\le A_\ell^r \ \forall r,\ell\right\}
\ee
is compact in $\Cr \infty Mk$. An important thing to note here is that every compact set in $\Cr \infty Mk$ is contained in a set of the form $\mathscr{A}^\infty$, while the analogous fact is not true when $r$ is finite.
\begin{proof}[Proof of theorem \ref{rproper}]
Let $Q\colon D\hookrightarrow M$ be the embedding of a compact subset $D\subset \R^m$. Then we can cover $D$ with a finite family of disks $D_1,\dots,D_N$ such that $Q(D_i)\subset \text{int}(Q_{\ell_i}(\D^m))$ for some $\ell_1,\dots,\ell_N$. It follows that there exists a constant $c>0$ such that $\|f\|_{Q,r+1}\le c\sum_{i=1}^N\|f\|_{Q_{\ell_i},r+1}$ for all $i=1,\dots, N$ and for all $f\in \Cr {r+1}Mk$.

 By applying Theorem \ref{Gaussinequality} to each $Q_\ell$, we get the inequality 
\be 
\begin{aligned}
\sup_n\E\left\{ \|X_n\|_{Q,r+1} \right\}
&\le c\sum_{i=1}^N\sup_n\E\left\{ \|X_n\|_{Q_{\ell_i},r+1} \right\}
\\
&\le c \sum_{i=0}^N C(Q_{\ell_i},r+1,k) \sqrt{\sup_{n}\|K_{X_n}\|_{Q_\ell\times Q_\ell,(r+2,r+2)}}
\\
&<+\infty.
\end{aligned}
\ee
This proves the implication \ref{itm:Ir} $\implies$ \ref{itm:IIr}.

By Prohorov's Theorem (see \cite[Theorem 5.2]{Billingsley}), to prove the second implication, it is sufficient to show that $\{[X_n]\}_n$ is tight in $\G(E^{r})$, i.e. that for every $\e>0$ there is a compact set $\mathscr{A}\subset E^{r}$, such that $\P(
X_n\in\mathscr{A})> 1-\e$ for any $n\in \N$.

Fix $\e>0$. By \ref{itm:IIr}, the number $A_\ell= {\sup_{n}\E\left\{\|X_n\|_{Q_\ell,r+1}\right\}}$  is finite for each $\ell\in\N$.
Thus, we can consider the compact subset $\mathscr{A}\subset \Cr rMk$ defined as follows,
\be 
\mathscr{A}=\left\{f\in \Cr r Mk \colon \|f\|_{Q_\ell,r+1}\le \frac{2^{\ell+1}}{\e}A_\ell, \ \forall \ell\in \N\right\}.
\ee
By subadditivity and Markov's inequality we have that for all $n\in \N$:
\be 
\begin{aligned}
\P\{X_n\notin \mathscr{A}\} & \le \sum_{\ell\in \N}\P\left\{\|X_n\|_{Q_\ell,r+1}>\frac{2^{\ell+1}}{\e}A_\ell\right\} \\
& \le \sum_{\ell\in \N} \frac{\e}{2^{\ell+1}}\cdot \frac{\E\{\|X_n\|_{Q_\ell,r+1}\}}{A_\ell} \\
& \le \sum_{\ell\in\N}2^{-(\ell+1)}\e=\e.
\end{aligned}
\ee
We conclude that $\{[X_n]\}_n$ is tight.
\end{proof}

%The proof of Theorem \ref{rproper} is essentially the same than that of Theorem \ref{infproper}, hence we give only the latter.
%\begin{proof}[Proof of theorem \ref{infproper}]
%By Prohorov's Theorem (see \cite[Theorem 5.2]{Billingsley}), it is sufficient to prove that $\{[X_n]\}_n$ is tight in $\G(E^{\infty})$, i.e. that if for every $\e>0$ there is a compact set $\mathscr{A}\subset E^{\infty}$, such that $\P(
%X_n\in\mathscr{A})\ge 1-\e$ for any $n\in \N$.
%
%Fix $\e>0$, and let $Q_\ell$ as above. By Theorem \ref{Gaussinequality} we have the inequality 
%\be 
%\E\left\{ \|X_n\|_{Q_\ell,r} \right\}\le C_\ell^r \sqrt{\sup_{n}\|K_{X_n}\|_{Q_\ell\times Q_\ell,(r+1,r+1)}}\le B_\ell^r 
%\ee
%for some positive constants $B_\ell^r, C_\ell^r>0$. By assumption, the constants $B_\ell^r$ exist finite. Define $A_\ell^r=(B_\ell^r)^{-1} 2^{r+\ell+2}$ and consider the compact set
%\be 
%\mathscr{A}=\left\{f\in \Cr \infty Mk \colon \|f\|_{Q_\ell,r}\le \frac{1}{\e}A_\ell^r \ \forall r,\ell\right\}.
%\ee
%By subadditivity and Markov's inequality we have that for all $n\in \N$:
%\be 
%\begin{aligned}
%\P\{X_n\notin \mathscr{A}\} & \le \sum_{r,\ell\in\N}\P\left\{\|X_n\|_{Q_i,r}>\frac{1}{\e}A^r_\ell\right\} \\
%& \le \sum_{r,\ell\in\N} \frac{B_\ell^r}{A_\ell^r}\e \\
%& = \sum_{r,\ell\in\N}2^{-(r+\ell+2)}\e=\e.
%\end{aligned}
%\ee
%We conclude that $\{[X_n]\}_n$ is tight.
%\end{proof}

\subsection{Proof that $\K^\infty$ is a closed topological embedding} We already know that $\K^\infty$ is injective and continuous. To prove that it is a closed topological embedding it is sufficient to show that $\K^\infty$ is proper: both $\G(E^{\infty})\subset \mathscr{P}(E^\infty)$ and $E^\infty$ are metrizable spaces, and a proper map between metrizable spaces is closed. 

Let $\mathscr{A}\subset \Cr \infty {M\times M}{k\times k}$ be a compact set; then for any $Q: D\hookrightarrow M$ embedding of a compact subset $D\subset \R^m$ and for every $r\in \N$, it holds
\be 
\sup_{K\in \mathscr{A}}\|K\|_{Q\times Q,r }<\infty .
\ee
Therefore Theorem \ref{infproper} implies that the closed subset $(\K^\infty)^{-1}(\mathscr{A})\subset\G(E^\infty)$ is also relatively compact, hence it is compact.
\subsection{Proof of Theorem \ref{thm:2}}By Theorem \ref{thm:1}, if $K_d\xrightarrow{\mathcal{C}^\infty}K$ then $\mu_d\nrw\mu$. Observe also that, by definition for every $A\subset E^\infty$:
\be \P(X\in A)=\mu(A)\quad \textrm{and}\quad \P(X_d\in A)=\mu_d(A).\ee
Consequently \eqref{eq:limitprob} follows from Portmanteau's theorem (see \cite[Theorem 2.1]{Billingsley}).
\subsection{Addendum: a ``counter-theorem''}
It is possible to improve Theorem \ref{Gaussinequality} in order to control $\E\{\|X\|_{Q,r}\}$ with a $(r+\a,r+\a)$ Holder norm of the covariance function, if the latter is finite for some $\a\in (0,1)$ (see \cite[Sec. A.9]{NazarovSodin2}). But there is no way to get such an estimate with $\a=0$, as the following example shows.
\begin{example}\label{escempio}
Let $D\subset \R^{m}$ compact with non empty interior. We now construct a sequence of smooth GRFs $X_n\in \g 0 D{}$, with $\|K_{X_n}\|_{D,0}\to 0$, such that
\be
\liminf_{n\to \infty}\E\{\|X_n\|_{D, 0}\}\ge 1 .
\ee

Let $I^{(n)}_1,\dots, I^{(n)}_{n^2}$ be disjoint open sets in $D$ (their size doesn't matter), containing points $x^{(n)}_1,\dots, x^{(n)}_{n^2}$. Let $\f^{(n)}_1,\dots, \f^{(n)}_{n^2}$ be smooth functions $\varphi_i^{(n)}:D\to [0,1]$ such that $\f^{(n)}_i$ is supported in $I^{(n)}_i$ and $\f^{(n)}_i(x^{(n)}_i)=1$ (see Figure \ref{fig:compact}).
Let $\gamma_i$ be a countable family of independent standard Gaussian random variables.
Let $a_n\in \R$ be the real number such that $\P\{|\gamma|>a_n\}=\frac{1}{n}$, for any $\gamma\sim N(0,1)$, hence $a_n \to +\infty$.
Define
\be 
X_n=\frac{1}{a_n}\sum_{i=1}^{n^2}\gamma_i \f^{(n)}_i \in \g 0D{}.
\ee
Then $K_{X_n}(x,y)=\frac{1}{a_n^2}\f^{(n)}_i(x)\f^{(n)}_j(y)$ for some $i=i_x,j=j_x$, thus $\|K_{X_n}\|_{D,0}\to 0$.

We can now estimate the probability that the $\mathcal{C}^0$-norm of $X_n$ is small by
\be\label{Pmisfacteq}
\begin{aligned}
\P\{\|X_n\|_{D,0}< 1\} &\le \P\left\{\max_{i=1,\dots, n^2}|X_n(x^{(n)}_i)|< 1\right\} \\
&= \P\{|\gamma|< a_n\}^{n^2} \\
&= \left(1-\frac 1n\right)^{n^2} \xrightarrow[{n\to \infty}]{} 0.
\end{aligned}
\ee
Consequently, by Markov's inequality
\be \label{Emisfacteq}
\liminf_{n\to \infty}\E\left\{\|X_n\|_{D,0}\right\}\ge \liminf_{n\to \infty}\P\left\{\|X_n\|_{D,0}\ge 1\right\}=1.
\ee
\begin{figure}\begin{center}
\includegraphics[width=0.7\textwidth]{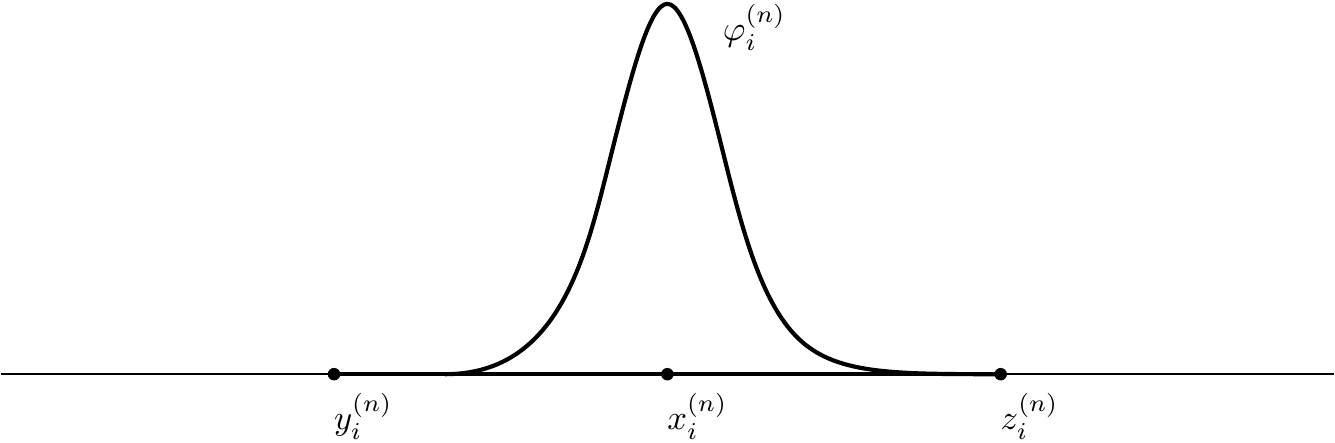}
	\caption{The function $\varphi_i^{(n)}$ from Example \ref{escempio} is supported on the interval $I_i^{(n)}=(y_i^{(n)}, z_i^{(n)})$ and takes value $1$ at $x_{i}^{(n)}.$}\label{fig:compact}
\end{center}
	\end{figure}
\end{example}

Note that the function $K(x,y)=0$ is the covariance function of the GRF $X_0$, which corresponds to the probability measure $\delta_{0}\in \G(E^0)$ concentrated on the zero function $0\in \Cr 0D{}$. Since $\P\left\{\|X_0\|_{D,0} \ge 1\right\}=0$, equation \eqref{Emisfacteq} proves also that $[X_n]$ does not converge to $[X_0]$ in $\G(E^0)$, even if $K_{X_n}\to K_{X_0}$ in $\Cr 0D{}$.

The previous Example \ref{escempio} can be generalized to prove the following result, which shows that the condition $r=\infty$ in the second part of the statement of Theorem \ref{thm:1} is necessary.

\begin{thm}\label{thm:counter}
If $r$ is finite, the map $(\mathcal{K}^r)^{-1}$ is not continuous.
\end{thm}
\begin{proof}
Construct $X_n\in \g 0D{}$ as in Example $\ref{escempio}$, with $D=[0,1]\subset \R$. Since $X_n$ is a sum of functions with compact support, we can as well consider $X_n$ as a random element of $\Cr 0\R{}$. So that $K_{X_n}\to 0$ in $\Cr 0{\R\times \R}{}$, because their support is contained in $D\times D$, but $X_n \centernot{\nrw} 0$. 

Let $c\notin D$ and let $Y_n$ be the GRF defined as
\be 
Y_n(\cdot)=\int_c^{(\cdot)}\int_c^{s_r}\dots \int_c^{s_2}X_n(s_1)ds_1 \dots ds_{r}.
\ee
Then $Y_n\in \g r\R{}$ (indeed $Y_n$ is a smooth GRF), and $\frac{d^r}{dx^r}Y_n=X_n$. Moreover,
\be 
\frac{d^{2r}}{dx^rdy^r}K_{Y_n}=K_{X_n}\to 0
\ee
in $\Cr 0{\R\times\R}{}$ and $K_{Y_n}=0$ in a neighborhood of $(c,c)$, therefore $K_{Y_n}\to 0$ in $\Cr {r,r}{\R\times\R}{}$, but $Y_n\not\nrw 0$.

Let $M$ be a smooth manifold of  dimension $m$. Denoting $(t,x)\in \R\times \R^{m-1}=\R^m$, we define a smooth function $\rho\colon \R^{m}\to [0,1]$ with compact support and such that $\rho=1$ in a neighborhood of the set $[0,1]\times \{0\}$. %\rho(t,0)=1$ for every $t\in [0,1]$. 
Let $j\colon \R^m\to M$ be any embedding and fix $v\in \R^k$. Define the transformation $T\colon \mathcal{C}^r(\R,\R)\to \Cr r Mk$ such that $f\mapsto g=Tf$, where
\be 
\begin{aligned}
g(j(t,x))=\rho(t,x)f(t)v && \\
g(p)=0 \quad \text{if }p\notin j(\R^m)
\end{aligned}
\ee
Since $\rho$ has compact support, $T$ is continuous for all $r\in\N\cup\{\infty\}$%; moreover $Y_n\in \mathcal{C}^\infty_c(\R,\R)$ almost surely
, so that $Z_n=TY_n$ is a well defined smooth GRF with compact support on $M$. Thanks to the continuity of $T$, we have that $K_{Z_n}\to 0$ in $\Cr {r,r}M{k\times k}$, but $Z_n \centernot{\nrw} 0$ in $\g rMk$ because $
\frac{d^r}{dt^r}Z_n(j(t,0))=X_n(t)
$
for every $t\in[0,1]$ and $X_n|_{D}\not\nrw 0$.
\end{proof}

%
%
%
%
%%%%%%%%%%%%%%%%%%%%%%%%%%%%%%%%%%%%%%%%%%%%%%%%%%%%%%%%%%%%%%%%%%%
%%%%%%%%%%%%%%%%%%%%%%%%%%%%%%%%%%%%%%%%%%%%%%%%%%%%%%%%%%%%%%%%%%%
%% PROOF OF THEOREM 3 %%
%%%%%%%%%%%%%%%%%%%%%%%%%%%%%%%%%%%%%%%%%%%%%%%%%%%%%%%%%%%%%%%%%%%
%%%%%%%%%%%%%%%%%%%%%%%%%%%%%%%%%%%%%%%%%%%%%%%%%%%%%%%%%%%%%%%%%%%
%
%
%
%

\section{Proof of Theorem \ref{thm:dtgrf:3}}\label{sec:PT3}

Given a Gaussian field $X=(X^1, \ldots, X^k)\in \mathcal{G}^{r}(M, \R^k)$ defined on a probability space $(\Omega, \mathfrak{S}, \PP)$, we consider the Hilbert space $\Gamma_X$ defined by:
\be \Gamma_X=\overline{\textrm{span}\{X^j(p),\,p\in M,\, j=1, \ldots, k\}}^{L^2(\Omega, \mathfrak{S}, \PP)}.\ee
Since $X$ is a Gaussian field, all the elements of $\Gamma_X$ are gaussian random variables (and viceversa).
By Lemma \ref{Ltwocont}, we know that the function $\gamma_X\colon M\to L^2\Prob$ defined as in equation \eqref{eq:gammaX} is of class $\mathcal{C}^r$, therefore we can define a linear map $\rho_X:\Gamma_X\to E^r$ by:
\be \rho_X(\gamma)=\E\left(X(\cdot)\gamma\right)=\left(\langle \gamma_X^1(\cdot), \gamma\rangle_{L^2(\Omega, \mathfrak{S}, \PP)},\dots,\langle \gamma_X^k(\cdot), \gamma\rangle_{L^2(\Omega, \mathfrak{S}, \PP)}\right)\ee
\begin{prop}\label{propo:injection}The map $\rho_X\colon \Gamma_X\to E^r$ is a linear, continuous injection.\end{prop}
\begin{proof}Let $\gamma\in\Gamma_X$ and assume that $\rho_X(\gamma)=0$. Then $\langle\gamma, X^j(p)\rangle_{L^2}=0$ for all $p\in M$ and $j=1,\ldots,k$, so that $\gamma \in \Gamma_X^\perp$, thus $\gamma=0$. This proves that $\rho_X$ is injective.

By linearity, it is sufficient to check continuity at $\gamma=0$.
Let $Q\colon D\hookrightarrow M$ be the embedding of a compact set $D\subset \R^m$. If $r$ is finite, we have
\begin{align}
\|\rho_X(\gamma)\|_{Q, r} &=\sup_{|\a|\le r, x\in D} |\E\left\{ \de_\a(X\circ Q)(x)\gamma\right\}|  \\
&\le \sqrt{k} \sup_{|\a|\le r, x\in D} \E\left\{|\de_\a(X\circ Q)(x)|^2\right\}^{\frac12}\|\gamma\|_{L^2} \\
&=\sqrt{k}\sup_{|\a|\le r, x\in D} \left(\sum_{j=1}^k \de_{(\a,\a)}(K_X^{j,j}\circ Q\times Q) (x,x)\right)^{\frac12}\|\gamma\|_{L^2}\\
&\le 
k\left(\|K_X\|_{Q\times Q, (r,r)}
\right)^{\frac12} \|\gamma\|_{L^2}.
\end{align}
Therefore $\lim_{\gamma \to 0}\|\rho_X(\gamma)\|_{Q, r}=0$ for every $Q$, hence $\rho_X$ is continuous. For the case $r=\infty$, it is sufficient to note that continuity with respect to $E^r$ for every $r$, implies continuity with respect to $E^\infty$.
\end{proof}
 
 \begin{prop}\label{prop:CamMa} The image of $\rho_X$ coincides with the Cameron-Martin space (see \cite[p. 44, 59]{bogachev}) of the measure $[X]$ and we denote it by $\mathcal{H}_X$. 
 \end{prop}
 \begin{proof}
 According to \cite[Lemma 2.4.1]{bogachev} $\mathcal{H}_X$ is the set of those $h\in E^r$ for which there exists a $T\in (E^r)^*$ such that
 \be\label{eq:CM}
 L(h)=\E\{T(X)L(X)\}, \quad \text{for all $L\in (E^r)^*$}
 \ee
 Observe that the map $T\mapsto T(X)$ defines a surjection $(E^r)^*\to\Gamma_X$, because every continuous linear functional $T\in (E^r)^*$ can be approximated by linear combinations of functionals of the form $\delta_p^j:f\mapsto f^j(p)$ (see Theorem \ref{thm:deltadense} in Appendix \ref{app:dual}). For the same reason, condition \eqref{eq:CM} is equivalent to the existence of $\gamma\in \Gamma_X$ such that
 \be 
 h^j(p)=\E\{\gamma X^j(p)\}\}, \quad \text{for all $p\in M$ and $j=1,\dots, k$}
 \ee
 that is, by definition, $h=\rho_X(\gamma)$. Thus $\mathcal{H}_X=\rho_X(\Gamma_X)$.
 \end{proof}

Observe that $\mathcal{H}_X$ contains all the functions $h_p^j=\rho_X(X^j(p))$ satisfying equation\eqref{suppbasiseq} in Theorem \ref{thm:dtgrf:3}.
Moreover, it carries the Hilbert structure induced by the map $\rho_X$, which makes it isometric to $\Gamma_X$. It follows that $\mathcal{H}_X$ is the Hilbert completion of the vector space span$\{h_p^j\colon p\in M,\ j=1,\dots,k\}$, endowed with the scalar product
\be 
\langle h_p^j,h_q^\ell\rangle_{\mathcal{H}_X}\doteq \left\langle X^j(p),X^\ell(q)\right\rangle_{L^2}=K_X^{j,\ell}(p,q).
\ee
Now, Theorem \ref{thm:dtgrf:3} follows from \cite[Theorem 3.6.1]{bogachev}:
\be\label{eq:sputoCM}
\text{supp}(X)=\overline{\mathcal{H}_X}^{\Cr rMk}.
\ee
In Appendix \ref{app:rgrf} (equation \eqref{eq:sputoCMapp}) the reader can find a proof of \eqref{eq:sputoCM} adapted to our language.
\begin{remark}
Note that the Hilbert space $\mathcal{H}_X$ depends only on $K_X$, thus it depends only on the measure $[X]$.
\end{remark}

%
%
%
%
%%%%%%%%%%%%%%%%%%%%%%%%%%%%%%%%%%%%%%%%%%%%%%%%%%%%%%%%%%%%%%%%%%%
%%%%%%%%%%%%%%%%%%%%%%%%%%%%%%%%%%%%%%%%%%%%%%%%%%%%%%%%%%%%%%%%%%%
%% PROOF OF THEOREM TRANSTHM2 %%
%%%%%%%%%%%%%%%%%%%%%%%%%%%%%%%%%%%%%%%%%%%%%%%%%%%%%%%%%%%%%%%%%%%
%%%%%%%%%%%%%%%%%%%%%%%%%%%%%%%%%%%%%%%%%%%%%%%%%%%%%%%%%%%%%%%%%%%
%
%
%
%

\section{Proof of Theorems \ref{transthm2} and \ref{thm:transthm}}\label{sec:transversality}
\subsection{Transversality}
We want to prove some results analogous to Thom's Transversality Theorem (see \cite[Section 3, Theorem 2.8]{Hirsch}) in our probabilistic setting. We first recall the definition of transversality. Let $f\colon M\to N$ be a smooth map, $W\subset N$ a submanifold and $K\subset M$ be any subset. Then we say that $f$ is transverse to $W$ on $K$ and write $f\transv_K W$, if and only if for every $x\in K\cap f^{-1}(W)$ we have:
\be 
df_x(T_xM)+T_{f(x)}W=T_{f(x)}N.
\ee
We will simply write $f\transv W$ if $K=W$.
%If $W\subset M$ is any subset, we say that $f$ is a submersion on $K$ and write $f\transv_K\{W\}$, if for any $x\in K\cap f^{-1}(W)$, it holds
%\be 
%df_x(T_xM)=T_{f(x)}N.
%\ee
We recall the following classical tool, usually called the \emph{Parametric Transversality Theorem}.
\begin{thm}[Section 3, Theorem 2.7 from \cite{Hirsch}]
\label{parametric transversality}
Let $g\colon P\times F\to N$ be a smooth map between smooth manifolds of finite dimension.
Let $W\subset N$ be a smooth submanifold and $K\subset P$ be any subset. If $g \transv_{K\times F} W$, then $g(\cdot, f)\transv_K W$ for almost every $f\in F$.
\end{thm}
%\begin{defi}
%We define the Jet evaluation map as 
%\be 
%j^r\colon M \times E^\infty  \to J^r, \qquad (p,f)\mapsto j^r_pf.
%\ee
%Where $J^r=J^r(M,\R^k)$.
%\end{defi}
In our context we prove the following infinite-dimensional, probablistic version of Theorem \ref{parametric transversality}.
\begin{thm}\label{transthm}
Let $F\subset E^r$ such that $F=\spt(X)$ for some $X\in \g r Mk$, with $r\in\N\cup\{\infty\}$. Let $P,N$ be smooth manifolds and $W\subset N$ a submanifold. Assume that $\Phi\colon P\times F\to N$ is a ``smooth''\footnote{Here by ``smooth'' we mean that:
\begin{enumerate}
    \item the map $\Phi$ is smooth when restricted to finite dimensional subspaces;
    %\item the map $\Phi$ is Gateaxu differentiable;
    %\comm{To discuss this point again}
    \item the linear map $(p, f, v)\mapsto D_{(p, f)}\Phi v= D_{(p, f)}\left(\Phi|_{\textrm{span}\{f,v\}}\right) v$ is continuous in all its arguments.
    \end{enumerate}} map such that $\Phi\transv W$. Then
\be 
\P\{\Phi(\cdot, X)\transv W\}=1.
\ee

\end{thm}
A particular case in which we can apply Theorem \ref{transthm} is when $P=M$,  $N=J^r=J^r(M,\R^k)$, $r=\infty$ and $\Phi$ is the jet-evaluation map
\be\label{eq:jr}
j^r\colon M \times E^\infty  \to J^r, \qquad (p,f)\mapsto j^r_pf.
\ee
It is straightforward to see that this map is ``smooth'' in the sense of the statement of Theorem \ref{transthm}.
\begin{proof}
(In order to simplify the notations, we denote by $\phi(X)$ the map $p\mapsto \Phi(p, X)$.)

First we show that we can assume $W$ to be compact (possibly with boundary). Indeed let $W=\cup_{k\in \N} W_k$, such that $W_k$ is compact. Then $\Phi\transv W_k$ for any k, and 
\be 
\P\{\phi(X)\transv W\}\ge 1- \sum_{k\in\N}\left(1-\P\{\phi(X)\transv W_k\}\right).
\ee

Moreover, we claim that it is sufficient to prove the following weaker statement.

\begin{stella}
 For all $p\in P$ and $x\in F$ there are neighborhoods $Q_p$ of $p$ in $P$ and $U_x$ of $x$ in $E^r$ such that: \be \P
\big\{\phi(X)\transv_{\overline{Q}_p}W\big|X\in U_x\big\}:=\frac{\P\left(\{\phi(X)\transv_{\overline{Q}_p}W\}\cap\{X\in U_x\}\right)}{\P\left(\{X\in U_x\}\right)}=1.\ee
\end{stella}

Assume that $(*)$ is true, then there exists a countable open cover of $P\times F$ of the form $Q_k\times U_l$ such that $\P\{\phi(X)\transv_{\overline{Q}_k}W|X\in U_l\}=1$, i.e. the probability that $X\in U_l$ and  $\phi(X)$ is not transverse to $W$ at some point $p\in \overline{Q}_k$ is zero. 
Thus
\be 
\P\{\phi(X)\not\transv W\}\le\sum_{l,k}\P\left\{\phi(X)\not\transv_{\overline{Q}_k}W,X\in U_l\right\}=0,
\ee
hence the claim is true.

Let us prove $(*)$. Let $p\in P$ and $x\in F$. Since $W\subset N$ is closed, if $\Phi(p,x)\notin W$, then $\Phi(q,\tilde{x})\notin W$ for all $q$ in a compact neighborhood $Q$ of $p$ and $\tilde{x}$ in some neighborhood $N_x$ of $x$ in $E^r$, so that, in particular $\P\{\phi(X)\transv_{Q}W|X\in N_x\}=1$. 

Assume now that $\Phi{(p,x)}=\theta\in W$, then by hypothesis we have that 
\be
D_{(p,x)}\Phi(T_pP+F) + T_\theta W=T_\theta N,
\ee
hence there is a finite dimensional space $F_0=\textrm{span}\{f_1,\dots,f_a\}\subset F$ such that 
\be
D_{(p,x)}\Phi\left(T_pP+F_0\right) + T_\theta W=T_\theta N.
\ee
Note that $F_0=T_xF_x$, where $F_x=x+\text{span}\{f_1,\dots,f_a\}$. Therefore 
$
\Phi|_{P\times F_x}\transv_{(p,x)} W
$
(here we are in a finite dimensional setting). Moreover, there is a compact neighborhood  $p\in Q\subset P$ and a $\e>0$ such that
\be \label{difftrasveq}
\Phi|_{P\times F_x}\transv_{Q\times D_\e} W.
\ee
where $D_\e=D_\e(x,f)=\{x+f_1u^1+\dots +f_au^a\colon u\in \R^a,\ |u|\le\e\}$.
Observe that the set of $(a+1)-$tuples $(x,f)=(x,f_1,\dots,f_a)\in F\times F^a$ for which \eqref{difftrasveq} holds (with fixed $\e$), form an open set, indeed the map 
\be 
\tau\colon F\times F^a\to \mathcal{C}^\infty(P\times \R^a,N),\qquad \tau(x,f):(p,u)\mapsto \Phi(p,x+fu)
\ee
is continuous and the set $\Theta=\{T\in \mathcal{C}^\infty(P\times \R^a,N)\colon T\transv_{Q\times D_\e}W\}$ is open in the codomain because $Q\times D_\e $ is compact and $W$ is closed (see \cite[p. 74]{Hirsch}); therefore 
\be 
\tau^{-1}(\Theta)=\{(x,f)\in F\times F^a\colon \text{\eqref{difftrasveq} holds}\}
\ee
is open. It follows that there is an open neighborhood $V_x$ of $x$ and an $h\in (\mathcal{H}_X)^a$ such that \eqref{difftrasveq} holds with $(\tilde{x},h)$ for any $\tilde{x}\in V_x$, indeed the Cameron-Martin space $\mathcal{H}_X$ is dense in $F$ (see equation \eqref{eq:sputoCM}). 

Define $\Lambda=\{e\in E^r \colon \phi(e)\transv_{Q}W\}$. 
By Theorem \ref{parametric transversality} we get that if $\tilde{x}\in V_x$, then $\phi(\tilde{x}+hu)\transv_{Q}W$, equivalently $(\tilde{x}+hu)\in \Lambda$, for almost every $|u|\le\e$. Denote by $1_\Lambda\colon E^r\to \{0,1\}$ the characteristic function of the (open) set $\Lambda$.
Using the Fubini-Tonelli theorem, we have
\be 
0=\int_{V_x}\left(\int_{\D^m_\e} 1-1_\Lambda(\tilde{x}+hu)du\right)d[X](\tilde{x})=\int_{\D^m_\e}\P\{X+hu\notin \Lambda, X\in V_x\}du.
\ee
hence $\P\{X+hu\in (V_x+hu) \- \Lambda\}=0$ for almost every $|u|\le\e$. Let $u$ be also so small that $x\in V_x+hu$. Then, taking $U_x=V_x+hu$, we have that $\P\{X+hu\in U_x\- \Lambda\}=0$. Since $hu\in \mathcal{H}_X$, the Cameron-Martin theorem (see \cite[Theorem 2.4.5]{bogachev}) implies that 
$[X]$ is absolutely continuous with respect to $[X+hu]$ and consequently $\P\{X\in U_x\- \Lambda\}=0$. In other words, $\P\{\phi(X)\transv_{Q}W|X\in U_x\}=1$, that proves $(*)$.
\end{proof}

We give know a criteria to check the validity of the hypothesis of Theorem \ref{transthm}, without necessarily knowing the support of $X$. Before that, let's observe that the canonical map $J^r:=J^r(M,\R^k)\to M$ is a smooth vector bundle over $M$ with fiber $J^r_p$, so that $T_{\theta}J_p^r$ is canonically identified with $J^r_p$ itself, for all $\theta\in J^r_p$. 
%Moreover, for any $f\in E^\infty$, the map $j^rf$ is a section of this bundle, hence there is a splitting
%\be 
%D_{(p,f)}j^r(T_pM)\oplus J_p^r=T_{j^r_pf} J^r.
%\ee
%which gives a natural isomorphism
%\be 
%\pi_{(p,f)}\colon \frac{T_\theta J^r}{D_{(p,f)}j^r(T_pM)}\cong J_p^r.
%\ee
\begin{prop}
Let $X\in\g rMk$ and $F=\spt(X)$. Let $W\subset J^r$ be a smooth submanifold and fix a point $p\in M$. The next conditions are related by the following chain of implications: $
\text{\ref{itm:transjetA}}\Longleftarrow \text{\ref{itm:transjetB}}\Longleftarrow \text{\ref{itm:transjetC}}\iff \text{\ref{itm:transjetD}}
$.
%\be
%\text{\ref{itm:transjetA}}\leftarrow \text{\ref{itm:transjetB}}\leftarrow \text{\ref{itm:transjetC}}\iff \text{\ref{itm:transjetD}}
%\ee.
\begin{enumerate}[$(I)$]
\item\label{itm:transjetA}  $j^r|_{M\times F}\transv_{\{p\}\times F} W$, where $j^r\colon M\times E^\infty\to J^r$ is the map defined in \eqref{eq:jr}; 
\item\label{itm:transjetB}  the vector space $\spt(j_p^rX)$ is transverse to  $(T_\theta W\cap T_\theta J^r_p)$ in $J^r_p$, for all $\theta \in j^r_p(F)\cap W$;
\item\label{itm:transjetC}  $\spt(j^r_pX)=J_p^r$;
\item\label{itm:transjetD}  given a chart of $M$ around $p$, the matrix below has maximal rank.
\be\label{covjeteq}
\left(\de_{(\a,\beta)}K_X(p,p)\right)_{|\a|,|\beta|\le r}.
\ee
\end{enumerate}
%Then 
%\be 
%(D)\iff (C),\quad (C)\implies (B)\quad \text{and}\quad (B) \implies (A).
%\ee
\end{prop}
\begin{proof}
\ref{itm:transjetB}$\implies$ \ref{itm:transjetA}. Let $f\in F$ such that $\theta=j^r_pf\in W$. Under the identification $T_\theta J^r_p\cong J^r_p$, mentioned above, we have
\be 
\left(D_{(p,f)}j^r\right)(0,g)= \frac{d}{dt}\Big|_0 j^r_p(f+tg)\cong j^r_pg,
\ee
so that, for any $x\in F$, $D_{(p,f)}j^r(T_xF)=j^r_p(F)=\spt(j^r_pX)$. Then for all $(p,f)\in (j^r)^{-1}(W)\cap M\times F$, we have
\be 
\begin{aligned}
D_{(p,f)}j^r\left(T_{(p,f)}(M\times F)\right)+T_\theta W &\supset 
D_{(p,f)}j^r(T_pM)+ \spt(j^r_pX)+T_\theta W\cap J^r_p=\\
& =D_{(p,f)}j^r(T_pM)+ J_p^r=\\ &=T_\theta J^r
 \end{aligned}
\ee
The last equality follows from the fact that the map $j^rf$ is a section of the bundle $J^r\to M$. 

\ref{itm:transjetC}$\implies$ \ref{itm:transjetB}.  Obvious.

\ref{itm:transjetC}$\iff$ \ref{itm:transjetD}. Any chart around $p$ defines a linear isomorphism 
\be 
J_p^r\to \R^{\{\a\colon |\a|\le r\}}, \quad j^r_pf\mapsto \left(\de_af(p)\right)_\a.
\ee
With this coordiante system, the covariance matrix of the Gaussian random vector $j^r_pX$, is exactly the one in $\eqref{covjeteq}$, hence the result follows from the fact that the random Gaussian vector $j^rf$ has full support if and only if its covariance matrix is nondegenerate.
\end{proof}
Given $X\in \g \infty Mk$, we can also consider it as an element of $ \g rMk$ such that $\P\{X\in \Cr\infty Mk\}=1$. We use the notation $\spt_{\mathcal{C}^r}(X)\subset E^r$ to denote the support of the latter, namely 
\be 
\spt_{\mathcal{C}^r}(X)=\overline{\mathcal{H}_X}^{\mathcal{C}^r}.
\ee
\begin{cor}\label{transcor}
Let $X\in \g \infty Mk$, such that $\mathrm{supp}_{\mathcal{C}^{r}}(X)=\Cr {r} Mk$. Then for every submanifold $W\subset J^r(M,\R^k)$, one has
\be 
\P\{j^rX\transv W\}=1.
\ee
\end{cor}
\begin{proof}
Clearly $X$ satisfies for every $p\in M$ condition \ref{itm:transjetC} of the proposition above, hence the hypotheses of Theorem \ref{transthm} are satisfied for every $W\subset J^r(M,\R^k)$.
\end{proof}

%\begin{subappendices}
\appendix
\section{The dual of $E^r$}\label{app:dual}
The purpose of this section is to fill the gap between Definition \ref{def:gauss} of a Gaussian measure on the space $E^r=\mathcal{C}^r(M,\R^k)$ and the abstract definition of Gaussian measures on topological vector spaces, for which we refer to the book \cite{bogachev}. As we already mentioned in Remark \ref{rem:bridge}, the two definitions coincide. In order to see this clearly, one simply have to  understand the topological dual of $E^r$, that is the space $(E^r)^*$ defined as follows.

Let $(E^r)^*$ be the set of all linear and continuous functions $T\colon E^r \to \R$, endowed with the weak-$*$ topology, namely the topology induced by the inclusion $(E^r)^*\subset \R^{E^r}$, when the latter is given the product topology.
\begin{lemma}\label{duallemma}
Let $T\in (E^r)^*$, with $r\in\N\cup\{\infty\}$. There exists a finite set $\mathscr{Q}$ of embeddings $Q\colon \mathbb{D}^m\hookrightarrow M$, a constant $C>0$ and a finite natural number $s\le r$, such that
\be\label{eq:dualiT}
|T(f)|\le C\max_{Q\in\mathscr{Q}}\|f\|_{Q,s},\footnote{The seminorm $\|\cdot\|_{Q,r}$ is defined as in equation \eqref{eq:seminormQR}.}
\ee
for all $f\in E^r$.
As a consequence, denoting $K=\cup_{Q\in\mathscr{Q}}Q(\mathbb{D})$, there is a unique $\hat{T}\in (\Cr sK k)^*$ such that $T(f)=\hat{T}(f|_{K})$ for all $f\in E^r$. 
\end{lemma}
Let $K\subset M$ be as in Lemma \ref{duallemma}. The vector space $\Cr sKk$ is, by definition, the image of the restriction map 
\be 
\Cr sMk \to \Cr 0K k, \quad f\mapsto f|_K.
\ee
%endowed with the quotient topology. In other words, a sequence $f_n\in\Cr sKk$ converges to $f$, if and only if there exists a sequence of extensions $g_n\in\Cr s Mk$, $g_n|_{K}=f_n$ such that $g_n\to g$ and $g|_K=f$.
Denote by $\Omega=\text{int}(K)\subset M$. Notice that the derivatives, of order less than $s$, of a function $f\in\Cr sKk$ are well defined and continuous at points of $\overline{\Omega}$, thus when $K=\overline{\Omega}$, we have a well defined continuous function \be j^sf\colon K\to J^s(K,\R^k)=\{j^s_pf\in J^s(M,\R^k)\colon p\in K\}.\ee In this case ($\overline{\text{int}(K)}=K$) we endow the space $\Cr sKk$ with the topology that makes Lemma \ref{Lemma:14} true with $M=K$. 
Such topology is equivalent to the one defined by the norm $\|\cdot\|_{K,s}$ below (it depends on $\mathscr{Q}$), with which $\Cr sKk$ becomes a Banach space:
\be\label{normKeq}
\|f\|_{K,s}=\max_{Q\in\mathscr{Q}}\|f\|_{Q,s}.
\ee
(Note that, if $K=\overline{\Omega}$, then $\|f\|_{K,s}$ depends only on $f|_{\Omega}$.)
\begin{remark}
When $M$ is an open subset $M\subset\R^m$ and $k=1$, the elements of $(E^\infty)^*$ are exactly the distributions with compact support (in the sense of \cite{schwartz1957}).
\end{remark}
\begin{proof}[Proof of Lemma \ref{duallemma}]
Let $Q_\ell\colon \mathbb{D}\hookrightarrow M$ be a countable family of embeddings such that $g_n\to 0$ in $E^r$ if and only if $\|g_n\|_{Q_\ell,s}\to 0$ for all $\ell\in\N$ and $s\le r$ (it can be constructed as in \eqref{prodembballeq}).
Assume that for all $n\in\N$ there is a function $f_n\in E^r$, such that
\be \label{dualabsurdeq}
|T(f_n)|> n\max_{\ell\le n}\|f_n\|_{Q_\ell,s_n},
\ee
where $s_n:=n$ if $r=\infty$, otherwise $s_n:=r$.
Then the sequence 
\be 
g_n=\frac{f_n}{n\max_{\ell\le N}\|f_n\|_{Q_\ell,s_n}}
\ee
converges to $0$ in $E^r$, indeed $\|g_n\|_{Q_\ell,s}\le \frac1N$ for any fixed $\ell\in \N$ and $s\le r$. Therefore, by the continuity of $T$, we get that $T(g_n)\to 0$. But $|T(g_n)|>1$ according to \eqref{dualabsurdeq}, so we get a contradiction. It follows that there exists $N$ such that for all $f\in E^r$ we have
\be
|T(f)|\le N\max_{\ell\le N}\|f\|_{Q_\ell,s_N}.
\ee
This proves the first part of the Lemma, with $\mathscr{Q}=\{Q_0,\dots,Q_N\}$, $C=N$ and $s=s_N$.

Define $\Omega=\text{int}(K)$. Note that, since $Q(\text{int}(\D^m))\subset \Omega$ for all $Q\in\mathscr{Q}$, if $p\in K\-\Omega$, then $p\in Q(\de\mathbb{D}^m)$ for some $Q\in \mathscr{Q}$ and therefore $p\in \overline{Q(\text{int}(\mathbb{D}^m))}\subset \overline{\Omega}$. This proves that $K=\overline{\Omega}$, hence $\Cr sKk$ is a Banach space with the norm \eqref{normKeq}. 

Let $f,g\in E^r$ be such that $f|_K=g|_K$, then \be 
\begin{aligned}
|T(f)-T(g)|=|T(f-g)|\le C\max_{Q\in\mathscr{Q}}\|f-g\|_{Q,s} =C\|f|_K-g|_K\|_{K,s}=0.
\end{aligned}
\ee
It follows that the function $L\colon \Cr rKk\to \R$ such that $L(f|_\Omega)=T(f)$ for all $f\in E^r$, is well defined and continuous with respect to the norm $\|\cdot\|_{K,s}$. Since $\Cr rKk$ is dense in $\Cr sKk$, there is a unique way to extend $L$ to a bounded linear functional on $\Cr sKk$, that we call $\hat{T}.$
\end{proof}
We recall the following classical theorem from functional analysis (see \cite[Theorem 1.54]{ambrofuscopalla}), which we can use to give a more explicit description of $(E^r)^*$.
\begin{thm}[Riesz's representation theorem]\label{thm:riesz}
Let $K$ be a compact metrizable space. Let $\mathcal{M}(K)$ be the Banach space of Radon measures on $K$ (on a compact set it is the set of finite Borel signed measures), endowed with the total variation norm.
Then the map 
\be 
\mathcal{M}(K)\to \left(\mathcal{C}(K)\right)^*, \quad \mu\mapsto \int_K(\cdot)d\mu 
\ee
is a linear isometry of Banach spaces.
\end{thm}
\begin{thm}\label{thm:Mloc}
Let $\mathcal{M}^r_{loc}$ be the set of all $T\in (E^r)^*$ of the form 
\be 
T(f)=\int_{\mathbb{D}^m}\de_\a(f^j\circ Q) d\mu,
\ee
for some embedding $Q\colon\mathbb{D}^m\hookrightarrow M$, some finite multi-index $\a\in\N^m$ such that $|\a|\le r$, some $j\in \{1,\dots,k\}$ and some $\mu\in \mathcal{M}(\mathbb{D}^m)$.
Then $(E^r)^*=\text{span}\{\mathcal{M}^r_{loc}\}$.
\end{thm}
\begin{proof}
Let $T\in (E^r)^*$ and let $\mathscr{Q}$, $s$, $K$,  $C$ and $\hat{T}$ defined as in lemma \ref{duallemma}. Consider the topological space
\be 
D=\mathbb{D}^m\times \mathscr{Q}\times \{\a\in \N^{m}\colon |\a|\le s\}\times\{1,\dots , k\}.
\ee
$D$ is homeomorphic to a finite union of disjoint copies of the closed disk, therefore it is compact and metrizable. There is a continuous linear embedding with closed image 
\be 
\begin{aligned}
\mathcal{J}^s\colon \Cr sKk \hookrightarrow  \mathcal{C}(D), \quad
\mathcal{J}^sf(u,Q,\a,j)=\de_\a (f^j\circ Q)(u).
\end{aligned}
\ee
Indeed $\|\mathcal{J}^sf\|_{\mathcal{C}(D)}\le \|f\|_{K,s}\le \sqrt{k}\|\mathcal{J}^sf\|_{\mathcal{C}(D)}$, if $\|\cdot\|_{K,s}$ is defined as in \eqref{normKeq}. By identifying $\Cr sKk$ with its image under $\mathcal{J}^s$, we can extend $\hat{T}$ to the whole $\mathcal{C}(D)$, using the Hahn-Banach theorem and the extension can then be represented by a Radon measure $\mu\in \mathcal{M}(D)$, by the Riesz theorem \ref{thm:riesz}. 

Denote by $\mu_{Q,\a,j}\in \mathcal{M}(\mathbb{D}^m)$ the restriction of $\mu$ to the connected component $\mathbb{D}^m\times\{Q\}\times \{\a\}\times\{j\}$. Let $T_{Q,\a,j}$ be the element of $\mathcal{M}^r_{loc}$ associated with $Q$, $\a$, $j$ and $\mu_{Q,\a,j}$.
Then, for all $f\in E^r$, we have 
\be \begin{aligned}
T(f)&=\hat{T}(f|_{K})\\
&=\int_D \mathcal{J}^sfd\mu \\
&=\sum_{Q\in\mathscr{Q}, |\a|\le s, j=1,\dots,k}\int_{\mathbb{D}^m\times\{Q\}\times \{\a\}\times\{j\}}\mathcal{J}^sf d\mu \\
&=\sum_{Q\in\mathscr{Q}, |\a|\le s, j=1,\dots,k}\int_{\mathbb{D}^m}\de_\a(f^{j}\circ Q)d\mu_{Q,\a,j}\\
&=\sum_{Q\in\mathscr{Q}, |\a|\le s, j=1,\dots,k}T_{Q,\a,j}(f),
\end{aligned}
\ee
Therefore $T$ is the sum of all the $T_{Q,\a,j}$, thus $T\in \text{span}\{\mathcal{M}^r_{loc}\}$.
\end{proof}
We are now in the position of justifying Remark \ref{rem:bridge}. First, observe that manifold $M$ is topologically embedded in $(E^r)^*$, via the natural association $p\mapsto \delta_p$. We denote by $\delta_M\subset (E^r)^*$ the image of the latter map (it is a closed subset). From this it follows that any abstract Gaussian measure on $E^r$ is also a Gaussian measure in the sense of Definition \ref{def:gauss}. The opposite implication is a consequence of the following Lemma, combined with the fact that the pointwise limit of a sequence of Gaussian random variable is Gaussian.
\begin{cor}\label{thm:deltadense}
$(E^r)^*=\overline{\text{span}\{\delta_M\}}$.
\end{cor}
\begin{proof}
By Theorem \ref{thm:Mloc}, it is sufficient to prove that $\mathcal{M}^r_{loc}\subset \overline{\text{span}\{\delta_M\}}$. To do this, we can restrict to the case $M=\mathbb{D}^m$, $Q=$id and $k=1$.

Observe  that any functional of the type $\delta_p\circ\de_\a$ belongs to $\overline{\text{span}\{\delta_M\}}$. This can be proved by induction on the order of differentiation $|\a|$: if $|\a|=0$ there is nothing to prove, otherwise we have 
\be 
\delta_u\circ\frac{\de}{\de u^j}\circ\de_{\a}= \lim_{n\to \infty}n\left(\delta_{u+\frac{1}{n}e_j}\circ\de_\a-\delta_u\circ\de_\a\right)\in \overline{\text{span}\{\delta_M\}}.
\ee
 Note also that any $T^r\in\mathcal{M}^r_{loc}$ is of the form $T^0\circ \de_\a$ for some $T^0\in \mathcal{M}^0_{loc}$ and $|\a|\le r$ and, together with the previous consideration, this implies that it is sufficient to prove the theorem in the $r=0$, so that we can conclude with the following lemma
 \begin{lemma}
Let $K$ be a compact metric space. The subspace $\text{span}\{\delta_K\}$ is sequentially dense (and therefore dense) in $\mathcal{M}(K)$, with respect to the weak-$*$ topology on $\mathcal{M}(K)=\mathcal{C}(K)^*$.
\end{lemma}
Let $\mu$ be a Radon measure on $K$. Define for any $n\in\N$ a partition $\{A^n_i\}_{i\in I_n}$ of $K$ in Borel subsets of diameter smaller that $\frac1n$ and let $a^n_i\in A^n_i$. Define $t_n\in \text{span}\{\delta_K\}$ as
\be 
t_n=\sum_{i\in I_n}\mu(A^n_i)\delta_{a^n_i}.
\ee
Given $f\in \mathcal{C}(K)$, we have 
\be 
\begin{aligned}\label{heinecantoreq}
\left|\int_{K}f d\mu-t_n(f)\right|&\le \sum_{i\in I_n}\left|\int_{A^n_i}f-f(a^n_i)d\mu\right| \\
&\le |\mu|(K)\sup_{|x-y|\le \frac1n}|f(x)-f(y)|.
\end{aligned}
\ee
By the Heine-Cantor theorem, $f$ is uniformly continuous on $K$, hence the last term in \eqref{heinecantoreq} goes to zero as $n\to \infty$. Therefore, for every $f\in \mathcal{C}(K)$ we have that $t_n(f)\to\int_Kfd\mu$, equivalently: $t_n\to \mu$ in the weak-$*$ topology.
\end{proof}
We conclude this Appendix with an observation on the case $r=\infty$.
\begin{prop}
Let $T\in \mathcal{M}^\infty_{loc}$. Then the associated measure $\mu$ can be assumed to be of the form $\rho du$ for some $\rho\in L^{\infty}(\mathbb{D})$.
\end{prop}
\begin{proof}
Let $T\in\mathcal{M}^\infty_{loc}$ be associated with $Q$, $\a$, $\mu$. 
It is not restrictive to assume $M=\mathbb{D}^m$, $Q=id$ and $k=1$.
Let us extend $T$ to the space $\mathcal{C}^\infty_c(\R^m)$ by declaring
\be 
T(\f)=T(\f|_{\mathbb{D}^m})=\int_{\mathbb{D}^m}\de_\a\f d\mu.
\ee
Let $e\in\N^m$ be the multi-index $e=(1,\dots,1)$. Note that for all $\f\in\mathcal{C}^\infty_c(\R^m)$
\be 
\max_{u\in \R^m} \|\de_\a \f(u)\|\le \int_{\R^m}\left|\de_{\a+e}(u) \f\right|du.
\ee
Define $V\subset L^1(\R^m)$ as $ 
V=\{\de_{\a+e}\f\colon \f\in \mathcal{C}^\infty_c(\R^m)\}$, and let $\lambda\colon V\to\R$ be the liner function defined by $\lambda(\de_{\a+e}\f)=T(\f)$. Then $\lambda$ is a (well defined) linear and bounded functional on $(V,\|\cdot\|_{L^1})$, since 
\be 
\begin{aligned}
|\lambda(\de_{\a+e}\f)|&=|T(\f)|
=|\int_{\mathbb{D}^m} \de_\a \f d\mu|
\le |\mu|(\mathbb{D}^m)\max_{\R^m} \|\de_\a \f\|
\le |\mu|(\mathbb{D}^m)\|\de_{\a+e}\f\|_{L^1}.
\end{aligned}
\ee
The Hahn-Banach theorem, implies that $\lambda$ can be extended to a continuous linear functional $\Lambda$ on the whole space $L^1(\R^m)$ and hence it coincides, as a distribution, with a function $\rho\in L^\infty(\D^m)=L^1(\D^m)^*$.
In particular, for all $\f\in \mathcal{C}^\infty_c(\R^m)$, we have that 
\be 
\begin{aligned}
T(\f)&= \lambda(\de_{\a+e} \f)
&= \int_{\R^n}\de_{\a+e}\f(u)\rho(u) du.\\
\end{aligned}
\ee
\end{proof}
%
%
%
%%%%%%%%%%%%%%%%%%%%%%%%%%%%%%%%%%%%%%%%%%%%%%%%%%%%%%%%%%%%%%%%%%%%
%%%%%%%%%%%%%%%%%%%%%%%%%%%%%%%%%%%%%%%%%%%%%%%%%%%%%%%%%%%%%%%%%%%%
%boh
%%%%%%%%%%%%%%%%%%%%%%%%%%%%%%%%%%%%%%%%%%%%%%%%%%%%%%%%%%%%%%%%%%%%
%%%%%%%%%%%%%%%%%%%%%%%%%%%%%%%%%%%%%%%%%%%%%%%%%%%%%%%%%%%%%%%%%%%%
%
%
%
\section{The representation of Gaussian Random Fields}\label{app:rgrf}
%The easiest example of GRF is a field of the type $X=\xi_1 f_1 +\dots +\xi_n f_n$, where for $i=1, \ldots, n$ each $\xi_i$ a real Gaussian variables, $f_i\in \Cr rMk$ and the $\xi_i$ are independent. A slightly more general example is a series
%\be X=\sum_{n=0}^\infty \xi_n f_n\ee
%which is narrowly convergent (i.e. such that $s_n=\sum_{j\leq n}\xi_j X_j\nrw X$ for some $X\in \mathcal{G}^r(M, \R^k)$). 

The purpose of this section is to prove that every GRF $X\in \g rMk$ (where $r$ may be infinite) is of the form \eqref{eq:KarLo}. This fact is well-known in the general theory of Gaussian measures on Fréchet spaces (see \cite{bogachev}) as the Karhunen-Loève expansion. Here we present and give a proof of such result using our language, with the scope of making the exposition more complete and self-contained. Our presentation is analogous to that in the Appendix of \cite{NazarovSodin2} which, however, treats only the case of continuous Gaussian random functions, namely GRFs in $\g 0M{}$.
%usually called Karhunen-Lo\`eve expansion \comm{find a good reference for this}. 

Given a Gaussian random field $X\in \mathcal{G}^r(M, \R^k)$, we define its Cameron-Martin Hilbert space $\mathcal{H}_X\subset E^r$ to be the image of the map $\rho_X$, as we did in Section \ref{sec:PT3}.
\be \rho_X\colon \Gamma_X=\overline{\textrm{span}\{X^j(p),\,p\in M\}}^{L^2(\Omega, \mathfrak{S}, \PP)} \to E^r
\ee
\be \rho_X(\gamma)=\E\left(X(\cdot)\gamma\right)=\left(\langle X^1(\cdot), \gamma\rangle_{L^2(\Omega, \mathfrak{S}, \PP)},\dots,\langle X^k(\cdot), \gamma\rangle_{L^2(\Omega, \mathfrak{S}, \PP)}\right).\ee
(This is consistent with the more abstract theory from \cite{bogachev} because of Proposition \ref{prop:CamMa}.)
Note that $\mathcal{H}_X$ is separable, since $M$ is, hence it has a countable Hilbert-orthonormal basis $\{h_n\}_{n\in \N}$, corresponding via $\rho_X$ to a Hilbert-orthonormal basis $\{\xi_n\}_{n\in \N}$ in $\Gamma_X$. This means that for any $p$ and $j$, one has $h_n^j(p)=\langle X^j(p),\xi_n\rangle$, namely that $h_n^j(p)$ is precisely the $n^{th}$ coordinate of $X^j(p)$ with respect to the basis $\{\xi_n\}_{n\in\N}$. In other words: 
\be \label{repreq}
X(p)=\lim_{n\to \infty} \sum_{m\leq n}\xi_m h_m(p),
\ee
where the limit is taken in $L^2\Prob^k$. 
%Note that the $\gamma_n$ are Gaussian independent, by construction.
In particular, since the $L^2$ convergence of random variables implies their convergence in probability, we have that
\be\label{eq:rep}
\lim_{n\to \infty}\P\left\{\left|\sum_{m>n}\xi_m h_m(p)\right|>\e\right\}=0.
\ee 
\begin{thm}[Representation theorem]\label{thm:rep} Let $X\in \mathcal{G}^r(M, \R^k)$ be a GRF, with $r\in\N\cup\{\infty\}$. For every Hilbert-orthonormal basis $\{h_n\}_{n\in \N}$ of $\mathcal{H}_X$, there exists a sequence $\{\xi_n\}_{n\in N}$ of independent, standard Gaussians such that the series $\sum_{n\in \N}\xi_nh_n$ converges\footnote{Given a sequence $\{x_n\}_{n\in \N}\subset E$, the sentence ``the series $\sum_{n\in \N}x_n$ converges in $E$ to $x$'' means that $s_N=\sum_{n\leq N}x_n$ converges in $E$ to $x$ as $N\to \infty$.} in $E^r$ to $X$ almost surely.
\end{thm}
We will a convergence criterion for a random series (see Theorem \ref{converepr}). It essentially follows from the Ito-Nisio theorem, which we recall for the reader's convenience.
\begin{thm}[Ito-Nisio]\label{itonisiothm}
Let $E$ be a separable real Banach space. Let $M\subset E^*$ be such that the family of sets of the form $\{f\in E \ | \ \langle p,f \rangle \in A\}$, with $A\in \mathcal{B}(\R)$ and $p\in M$, generates the Borel $\sigma$-algebra of $E$. Let $\{x_n\}_{n\in \N}$ be independent symmetric random elements of $E$, define
\be 
X_n=\sum_{m\leq n} x_m.
\ee
Then the following statements are equivalent:
\begin{enumerate}
\item $X_n$ converges almost surely;
\item $\{X_n\}_{n\in \N}$ is tight in $\mathscr{P}(E)$;
\item There is a random variable $X$ with values in $E$ such that $\langle p, X_n \rangle \to \langle p, X \rangle$ in probability for all $p\in M$.
\end{enumerate}
 
\end{thm}
\begin{remark}
In the original paper \cite{ItoNisio}, the theorem is stated with the hypothesis that $M=E^*$, but the same proof works in the slightly weaker assumptions of Theorem \ref{itonisiothm}.
\end{remark}

\begin{thm}\label{converepr}
Let $x_n\in \g rMk$, for all $n\in \N$. Assume that $x_n$ are independent and consider the sequence $X_n$ of GRFs defined as
\be X_n=\sum_{j\leq n} x_j.\ee 
The following conditions are equivalent.
\begin{enumerate}
\item $X_n$ converges in $\Cr rMk$ almost surely.
\item Denoting by $\mu_n$ the measure associated to $X_n$, we have that $\{\mu_{n}\}_{n\in \N}$ is relatively compact in $\G(E^r)$. 
\item There is a random field $X$ such that, for all $p\in M$, the sequence $\{X_n(p)\}_{n\in \N}$ converges in probability to $X(p)$.
\end{enumerate}
\end{thm} 

\begin{proof}We prove both that $(1)\iff (2)$ and $(3)\iff(1)$. We repeatedly use the fact that a.s. convergence implies convergence in probability, which in turn implies convergence in distribution (narrow convergence).

$(1)\Rightarrow (2)$ This descends directly from the fact that almost sure convergence implies narrow convergence.

$(1)\Rightarrow (3)$ This step is also clear, since the almost sure convergence of $X_n$ to some random field $X$ implies that for any $p\in M$ the sequence of random vectors $X_n(p)$ converges to $X(p)$ almost surely and hence also in probability.

$(3)\Rightarrow (1)$ and $(2)\Rightarrow (1)$
Let $Q_\ell\colon D\hookrightarrow M$, be a countable family of embeddings of the compact disk, as in \eqref{prodembballeq}. Note that if $\{X_n\}_n$ is tight in $\g rMk$ (i.e. $\mu_n$ is tight in $\G(E^r)$), then $\{X_n\circ Q_\ell\}_n$ is tight in $\g rDk$. Moreover,
if $X_n\circ Q_\ell \to X\circ Q_\ell$ almost surely in $\Cr rDk$, for every $\ell\in \N$, then $X_n\to X$ almost surely in $\Cr rMk$. Therefore it is sufficient to prove the theorem in the case $M=D$. For analogous reasons, we can assume that $r$ is finite.

The topological vector space $E=\Cr rDk$ has the topology of a separable real Banach space, with norm 
\be 
\|\cdot\|_E=\|\cdot \|_{\text{id}_D, r}.
\ee
Since the $\sigma$-algebra $\mathcal{B}(\mathcal C^{r}(M, \R^k))$ is generated by sets of the form $\{f: f(p)\in A\}$, where $p\in M$ and $A\subset \R^k$ is open and since Gaussian variables are symmetric, we can conclude applying the Ito-Nisio Theorem \ref{itonisiothm} to the sequence $X_n$ of random elements of $E^r$.
\end{proof}
\begin{proof}[Proof of Theorem \ref{thm:rep}] Let $\{h_{n}\}_{n\in \N}$ be a Hilbert orthonormal basis for $\mathcal{H}_X$ and set $\xi_n=\rho_X^{-1}(h_n)$ (it is a family of independent, real Gaussian variables). From equation \eqref{eq:rep} we get that for every $p\in M$ and $j=1,\dots,k$ we have convergence in probability for the series:
\be
\label{eq:ppp}X(p)=\lim_{n\to \infty} \sum_{m\leq n}h_n(p) \xi_n.
\ee
Then, the a.s. convergence of the above series in $\mathcal{C}^r(M, \R^k)$ follows from point (1) of Theorem \ref{converepr}.
\end{proof}
\subsection{The support of a Gaussian Random Field}
By definition (see equation \eqref{eq:defsupport}), the support of $X\in\g rMk$ has the property that if it intersects an open set $U\subset E^r$, then $\P\{X\in U\}>0$. The following proposition guarantees that the converse is also true, namely that if $\P\{X\in U\}>0$, then $U\cap \spt(X)\neq \emptyset$.

\begin{prop}
The support of $X\in \g rMk$ is the smallest closed set $C\subset E^r$ such that $\P\{X\in C\}=1$. 
\end{prop}

\begin{proof}
By definition we can write the complement of $\textrm{supp}(X)$ as
\be
\left(\text{supp}(X)\right)^c=\bigcup\{U\subset \mathcal{C}^{r}(M, \R^k) \text{ open such that }\P\{X\in U\}=0\}.
\ee

Consequently $\textrm{supp}(X)$ equals the intersection of all closed sets $C\subset E^r$ such that $\P\{X\in C\}=1$, hence it is closed. Since $E^r$ is second countable the union above and the resulting intersection of closed sets can be taken over a countable family, so that $\P\{X\in \text{supp}(X)\}=1$.
\end{proof}
\begin{remark}
Assume that $X_n\nrw X \in \g rMk$ and recall Portmanteau's theorem (see \cite[Theorem 3.1]{Billingsley}). Then, for any open set $U\subset \Cr rMk$ such that $U\cap \text{supp}(X)\neq \emptyset$, there is a constant $p_U=\frac12 \P\{X\in U\}>0$ such that for $n$ big enough, one has
\be 
\P\{X_n\in U\}\ge p_U.
\ee
In particular, this implies that
\be 
\text{supp}(X)\subset \bigcap_{n_0}\overline{\left(\bigcup_{n\ge n_0}\text{supp}(X_n)\right)}=\limsup_{n\to \infty}\text{supp}(X
_N).
\ee
\end{remark}

\begin{thm}[The support of a Gaussian random map]\label{thm:sputo}Let $X\in \mathcal{G}^{r}(M, \R^k)$. Let $\{f_n\}_{n\in \N}\subset E^r$ and consider a sequence $\{\xi_n\}_{n\in\N}$ of independent, standard Gaussians. Assume that the series $\sum_{n\in \N}\xi_n f_n$ converges in $E^r$ to $X$ almost surely. Then
\be \mathrm{supp}(X)=\overline{\mathrm{span}\{f_n\}_{n\in \N}}^{E^r}.\ee
\end{thm}

\begin{proof}
We start by observing that  $X\in \overline{\text{span}\{f_n\}_n}$ with $\P=1$, thus the first inclusion ``$\subset$'' is proved.
%\be \text{supp}(X)\subset \overline{\text{span}\{f_n\}_{n\in \N}}^{\mathcal{C}^{r}(M, \R^k)}.\ee
Let now $c=\sum_{n=0}^{N_0}a_nf_n$ and let $U_c\subset \Cr rMk$ be an open neighborhood of $c$ of the form
\be 
U_c=\left\{ f\in \Cr rMk \colon \|f-c\|_{Q, r}<\e\right\}.
\ee
for some embedding $Q$.
Denote by $S_N=\sum_{n\leq N}\xi_n f_n$. Observe that if $N\ge N_0$, then $S_{N}-c\in \text{span}\{f_1\dots f_{N}\}$, which is a finite dimensional vector space, hence there is a constant $A_{N}>0$ such that $\|\sum_{n=0}^N a_nf_n\|_{Q,r}\le A_N\max\{|a_0|\dots |a_N|\}$. 
By the convergence in probability of $S_N$ to $X$, there exists $N>N_0$ so big that $\P\left\{\|X-S_{N}\|_{Q, r}\geq \frac{\epsilon}2\right\}<\frac12$.
Thus, setting $a_n=0$ for all $n>N_0$, we have:
\be
\begin{aligned}
\P\{X\in U_c\}&\ge \P\left\{\|X-S_{N}\|_{Q, r}<\frac{\e}2, \|S_{N}-c\|_{Q,r}<\frac\e 2\right\}\\
&\ge \P\left\{\|S_{N}-c\|_{Q,r}<\frac\e 2\right\}\frac12  \\
&\ge \left(\prod_{n=0}^{N}\P\left\{|\xi_n-a_n|<\frac{\e}{2A_{N}}\right\} \right)\frac12\\
&>0.
\end{aligned}
\ee

Every open neighborhood of $c$ in $\Cr rMk$ contains a subset of the form of $U_c$, therefore $c\in \text{supp}(X)$. Since $\text{supp}(X)$ is closed in $E^r$, we conclude.
\end{proof}
\begin{cor}\label{cor:sputoCMapp} 
Let $X\in\g rMk$ and let $\mathcal{H}_X\subset E^r$ be its Cameorn-Martin space.
\be\label{eq:sputoCMapp}
\spt(X)=\overline{\mathcal{H}_X}^{E^r}.
\ee
\end{cor}
\begin{proof}
It is a consequence of Theorems \ref{thm:rep} and \ref{thm:sputo}.
\end{proof}
%
%
%%%%%%%%%%%%%%%%%%%%%%%%%%%%%%%%%%%%%%%%%%%%%%%%%%%%%%%%%%%%%%%%%%%%%%%%%%%%
%%%%%%%%%%%%%%%%%%%%%%%%%%%%%%%%%%%%%%%%%%%%%%%%%%%%%%%%%%%%%%%%%%%%%%%%%%%%
%
%
%
%\end{subappendices}

\bibliographystyle{plain}

\bibliography{Differential_Topology_Of_Gaussian_Random_Fields}

\begin{thebibliography}{10}

\bibitem{AdlerTaylor}
R.~J. Adler and J.~E. Taylor.
\newblock {\em Random fields and geometry}.
\newblock Springer Monographs in Mathematics. Springer, New York, 2007.

\bibitem{ambrofuscopalla}
L.~Ambrosio, N.~Fusco, and D.~Pallara.
\newblock {\em Functions of Bounded Variation and Free Discontinuity Problems}.
\newblock Oxford Science Publications. Clarendon Press, 2000.

\bibitem{Billingsley}
Patrick Billingsley.
\newblock {\em Convergence of probability measures}.
\newblock Wiley Series in Probability and Statistics: Probability and
  Statistics. John Wiley \& Sons, Inc., New York, second edition, 1999.
\newblock A Wiley-Interscience Publication.

\bibitem{bogachev}
V.I. Bogachev and American~Mathematical Society.
\newblock {\em Gaussian Measures}.
\newblock Mathematical surveys and monographs. American Mathematical Society,
  1998.

\bibitem{EdelmanKostlan95}
Alan Edelman and Eric Kostlan.
\newblock How many zeros of a random polynomial are real?
\newblock {\em Bull. Amer. Math. Soc. (N.S.)}, 32(1):1--37, 1995.

\bibitem{Hirsch}
Morris~W. Hirsch.
\newblock {\em Differential topology}, volume~33 of {\em Graduate Texts in
  Mathematics}.
\newblock Springer-Verlag, New York, 1994.
\newblock Corrected reprint of the 1976 original.

\bibitem{ItoNisio}
Kiyosi Ito and Makiko Nisio.
\newblock On the convergence of sums of independent banach space valued random
  variables.
\newblock {\em Osaka J. Math.}, 5(1):35--48, 1968.

\bibitem{counterSardKupka}
Ivan Kupka.
\newblock Counterexample to the morse-sard theorem in the case of
  infinite-dimensional manifolds.
\newblock {\em Proceedings of the American Mathematical Society},
  16(5):954--957, 1965.

\bibitem{NazarovSodin2}
F.~Nazarov and M.~Sodin.
\newblock Asymptotic laws for the spatial distribution and the number of
  connected components of zero sets of {G}aussian random functions.
\newblock {\em Zh. Mat. Fiz. Anal. Geom.}, 12(3):205--278, 2016.

\bibitem{Parth}
K.R. Parthasarathy.
\newblock {\em Probability Measures on Metric Spaces}.
\newblock Ams Chelsea Publishing. Academic Press, 2005.

\bibitem{schwartz1957}
L.~Schwartz.
\newblock {\em Th{\'e}orie des distributions}.
\newblock Number v. 1-2 in Actualit{\'e}s scientifiques et industrielles.
  Hermann, 1957.

\bibitem{smalesard}
S.~Smale.
\newblock An infinite dimensional version of sard's theorem.
\newblock {\em American Journal of Mathematics}, 87(4):861--866, 1965.

\end{thebibliography}

\end{document}